\documentclass{article}

\usepackage{amssymb}
\usepackage{amsmath}
\usepackage{amsthm}
\usepackage{amsfonts}
\usepackage{epsfig}

\usepackage[english]{babel}

\newtheorem{defi}{Definition}[section]
\newtheorem{prop}[defi]{Proposition}

\newtheorem{lem}[defi]{Lemma}
\newtheorem{theo}[defi]{Theorem}
\newtheorem{rem}[defi]{Remark}
\newtheorem{hyp}{Assumption}

\makeatletter

\@addtoreset{equation}{section}
\makeatother

\newcommand{\dx}{\mathrm d}
\newcommand{\R}{\mathbb R}

\newcommand{\Z}{\mathbb Z}
\newcommand{\Rdd}{{\mathbb R^{d-1}}}
\newcommand{\E}{\mathbb E}

\newcommand{\LL}{\mathbb L}
\newcommand{\HH}{\mathbb H}
\newcommand{\T}{\mathbb T}
\newcommand{\Td}{{\T\times\R^{d-1}}}
\newcommand{\C}{{\mathcal C_T}}
\newcommand{\z}{^{2\hdots d}}
\newcommand{\phiepss}{\varphi_{\tilde\eta}}
\newcommand{\phieps}{\varphi_\eta}

\def\acknow{ {\underline {Acknowledgement} : We warmly thank Mathias
    Rousset for his careful reading of an early version of the manuscript.}}

\title{Existence, uniqueness and convergence of a particle approximation for the Adaptive
Biasing Force process}
\author{Benjamin~Jourdain$^1$, Tony~Leli\`evre$^2$, Rapha\"el~Roux$^2$}
\date{October 2, 2009}

\begin{document}
\maketitle

{$^1$
Universit\'e Paris-Est,
CERMICS, 6 et 8 avenue Blaise Pascal, 77455 Marne-La-Vall\'ee Cedex 2, France}

{$^2$
Universit\'e Paris-Est, CERMICS, Project-Team MICMAC ENPC-INRIA, 6 et 8
avenue Blaise Pascal, 77455 Marne-La-Vall\'ee Cedex 2, France}

\renewcommand{\thefootnote}{}
\footnotetext{\hspace{-6mm}This work was supported by the french National Research Agency (ANR)
under the programs ANR-08-BLAN-0218-03 BigMC and ANR-09-BLAN MEGAS.} 
\renewcommand{\thefootnote}{\arabic{footnote}}

\abstract{
We study a free energy computation procedure, introduced in
\cite{darve-pohorille-01,henin-chipot-04}, which relies on the long-time
behavior of a nonlinear stochastic
differential equation. This nonlinearity comes from a conditional
expectation computed with respect to one coordinate of the solution. The long-time convergence of the solutions to
this equation has been proved
in \cite{lelievre-rousset-stoltz-08}, under some existence and regularity assumptions.

In this paper, we prove existence and uniqueness under suitable conditions for the nonlinear equation, and
we study a particle approximation technique based on a Nadaraya-Watson estimator of
the conditional expectation. The particle system converges to the solution
of the nonlinear equation if the number of particles goes to infinity
and then the kernel used in the Nadaraya-Watson approximation tends to a
Dirac mass. 

We derive a rate for this convergence, and illustrate it by numerical
examples on a toy model.}
\clearpage

\section*{Introduction}

Free energy computations are an important problem in the field of
molecular simulation (see~\cite{chipot-pohorille-07}). The difficulty of those computations
lies in the fact that most dynamics in molecular simulations are highly
metastable: many free energy barriers prevent a good sampling. We study here the adaptive biasing force (ABF) method, which was introduced in
\cite{darve-pohorille-01,henin-chipot-04} to get rid of those metastabilities.

The typical problems one can think about are the study of a structural
angle in the conformation of a protein, or the measure of the evolution of a chemical reaction.
Mathematically, each configuration of the system is modelized by an
element of a high-dimensional state space $\mathcal D$,
typically an open subset of~$\R^d$, which is endowed with a probability measure, called
the canonical measure. This measure is given by $(\int_\mathcal D e^{-\beta V(x)}\dx x)^{-1}e^{-\beta V(x)}\dx x$, where~$V$ denotes the
potential energy undergone by the physical system, and $\beta$ is
proportional to the inverse of the temperature of the system.

For some $x$ in the state space, one is interested in a particular quantity, denoted by $\xi(x)$, $\xi$ being
assumed to be a smooth function from~$\mathcal D$ to the one-dimensional
torus $\T$. The quantity $\xi(x)$ has to be understood as a coarse-grained
information on the system, which is the relevant information for the
practitioner. In the examples above, $\xi(x)$ would be a structural
angle in a protein with conformation $x$,
or a number measuring the evolution of a chemical system in state $x$.\\
We call {\it free energy} the effective energy associated to the
quantity
$\xi(x)$, that is, the function $A(z)$ such that $e^{-\beta A(z)}\dx z$
is the image measure of the canonical measure by the function $\xi$.
Our objective is to compute numerically the function $A$.
When $\mathcal D=\R^d$, a naive method to do so is to simulate, for a given random variable
$X_0$ and an independent $\R^d-$valued Brownian motion $W$, the process defined by the
(overdamped) Langevin dynamics
\begin{equation}
\label{eq:langevin}\dx X_t=-\nabla V(X_t)\dx t+\sqrt{2\beta^{-1}}\dx W_t,
\end{equation}
 which, under some
regularity assumptions on the potential, is ergodic and admits the canonical measure as
unique invariant measure.
This approach appears to be untractable in practice, since the
convergence to equilibrium is very slow, due to multiple metastabilities
appearing in most problems: typically, a molecule moves
microscopically within times of order $10^{-15}$ seconds, while the
typical time scale of the macroscopic moves is of order $10^{-9}$ seconds.

The idea of the ABF method is to prevent the process $X_t$ from staying in
metastable states by introducing a biasing force which repel $X_t$ from
 the states where it stayed for too long a time. To do this, we use 
the following representation of $A$, that can be
deduced from the co-area formula (see~\cite{lelievre-rousset-stoltz-08}):
\begin{equation}\label{eq:representation_A'}
A'(z)=\E\left[F(X)|\xi(X)=z\right],
\end{equation}
where $X$ is a random variable distributed according to the canonical
measure, and $F$ is the function defined by
\begin{equation}\label{eq:def_F}
F(x)=\frac{\nabla\xi\cdot\nabla
  V}{|\nabla\xi|^2}-\frac1\beta\textrm{div}\left(\frac{\nabla\xi}{|\nabla\xi|^2}\right).
\end{equation}
The function $A'$ is called the {\it mean force}.
Actually, \eqref{eq:representation_A'} also holds when $X$ is
distributed according to the measure
$$\left(\int_\mathcal D e^{-\beta\left(V(x)+W\circ\xi(x)\right)}\dx
  x\right)^{-1}e^{-\beta\left(V(x)+W\circ\xi(x)\right)}\dx x,$$ which is
the canonical measure associated with the
biased potential $V+W\circ\xi$ where $W$ is any smooth function. 

Equation \eqref{eq:representation_A'} leads us to consider the following dynamics, which
should get rid of metastabilities for a well chosen $\xi$ since it
``flattens'' the energy landscape in the $\xi$ direction
(see~\cite{lelievre-rousset-stoltz-08} and Lemma~\ref{lem:chaleur} below for more precise statements):
\begin{align}\label{eq:EDS_ABF}
\begin{cases}
\dx X_t&=-\nabla \left(V-A_t\circ\xi-\beta^{-1}\ln(|\nabla\xi|^{-2})\right)(X_t)|\nabla\xi|^{-2}(X_t)\dx
t+\sqrt{2\beta^{-1}}|\nabla\xi|^{-1}(X_t)\dx W_t,\\
A_t'(z)&=\E\left[F(X_t)|\xi(X_t)=z\right].
\end{cases}
\end{align}
The second equality in \eqref{eq:EDS_ABF} shows that if $X_t$ is
distributed according to the canonical measure associated with the
potential $V-A\circ\xi$, then the biasing force $A_t'$ is actually the 
derivative $A'$ of the free energy, and the first equation in
\eqref{eq:EDS_ABF} consists in a Langevin dynamics associated to the
potential $V-A\circ\xi$. Consequently, the dynamics \eqref{eq:EDS_ABF}
admits a stationary point: $A_t'=A'$ and
$\textrm{Law}(X_t)=(\int_\mathcal D
e^{-V+A\circ\xi}\dx x)^{-1}e^{-V+A\circ\xi}\dx x$.

If we actually have convergence to this stationary state, we have a
method, that should be efficient ({\it i.e.} that should not see the metastabilities), to sample the
canonical measure up to a known perturbation $e^{A\circ\xi}.$ This
algorithm has thus two applications: it allows the computation of the
free energy $A$, and it can be used as an adaptative importance sampling
method for the canonical measure.

The long time behavior of Equation \eqref{eq:EDS_ABF} has been studied
in \cite{lelievre-rousset-stoltz-08}, where it has been proven that
for a sufficiently regular solution, one has, in some sense, an
exponential convergence to the stationary state, with a rate that is
better (for a well chosen $\xi$) than the rate of convergence to equilibrium for \eqref{eq:langevin}.

The practical difficulty in simulating \eqref{eq:EDS_ABF} is to compute the
conditional expectation, which is a highly nonlinear term. Stochastic
differential equations involving conditional expectations have already
been studied, in a case where the conditional
expectation is computed with respect to a random initial condition (see
\cite{talay-vaillant-03,tran-08}) or where the variable whose conditional
expectation is computed is fixed (see \cite{dermoune-03}). Our situation is much
more complex since both the conditioning and the conditioned variables
change with time and are affected by the
previous conditional expectations.

The same difficulty arises in Lagrangian stochastic models which are
commonly used in the simulation of turbulent flows (see \cite{bossy-jabir-talay-09}).
The main difference between the system studied in
\cite{bossy-jabir-talay-09} and~\eqref{eq:EDS_ABF} is that the authors
considers a Langevin dynamics with noise only on the velocity. The lack
of ellipticity then leads to additional difficulties. In our setting we
are able to derive a quantitative error estimate for the particle
discretization while this seems more difficult for Langevin dynamics.

 In this paper, we prove that existence and uniqueness hold for
Equation \eqref{eq:EDS_ABF} under suitable conditions, and we study an approximation of
$X_t$ by an interacting particle system (see Theorems~\ref{th:existence_loi} and \ref{th:approx_particule} below).

The paper is organized as follows. In Section \ref{enonce} we state our
main results.

Section \ref{regularite_unicite} is devoted to some uniqueness and
regularity results. More precisely, we prove that the time marginals of
a solution to Equation \eqref{eq:EDS_ABF} satisfy some partial
differential equation. Then,
under an integrability condition on the initial condition, we
prove uniqueness for the solutions to this equation, so that the
nonlinear term in \eqref{eq:EDS_ABF} is reduced to a bounded drift coefficient. We
thus prove pathwise uniqueness and uniqueness in distribution for the
solutions of \eqref{eq:EDS_ABF}.

Section \ref{dynamique_approchee} is devoted to existence results. More precisely, we introduce a regularization of
the dynamics \eqref{eq:EDS_ABF} involving two parameters $\alpha$ and $\varepsilon$, which is another nonlinear stochastic
differential equation whose nonlinearity is less singular. We prove that strong existence, pathwise uniqueness
and uniqueness in distribution hold for this equation and then we show
that the solutions to this stochastic differential equation converge to some process
which satisfies \eqref{eq:EDS_ABF} in the limit
$(\alpha,\varepsilon)\to(0,0)$, yielding strong existence. We also
prove that this convergence holds with rate $\mathcal O(\alpha+\sqrt\varepsilon)$.

In Section \ref{IPS} we introduce an interacting particle system to
approximate the regularized dynamics, and we prove a propagation-of-chaos
result for this particle system. We also derive a rate of convergence
for this propagation of chaos.

In Section \ref{numerique}, we illustrate the efficiency of the particle
approximation of the ABF
method and the rate of those convergences with some numerical examples
in small dimension.

\acknow

\section*{Notation}
We denote by $\T=\R/\Z$ the one dimensional torus, and for $x\in\R$, we
denote by $\{x\}\in\T$ its projection on $\T$. In the following, we will
work in two different domains $\mathcal D$: $\Td$ or $\T^d$. The case
$\mathcal D=\Td$ will be called the non compact case, and the case
$\mathcal D=\T^d$ will be called the compact case. For $x\in\R^d$, depending on the
case considered, we will also denote by $\{x\}$ the element of $\Td$
(resp.~$\T^d$) defined by $\{x\}=(\{x^1\},x^2,\hdots,x^d)$ (resp. $\{x\}=(\{x^1\},\hdots,\{x^d\})$).

In the following, we will call ``function defined on
$\T$'' (resp. on $\Td$, resp. on~$\T^d$), a $\mathbb Z-$periodical
(resp. $\mathbb Z-$periodical in the first
coordinate, resp. $\Z^d$-periodical) function defined on $\R$ (resp on
$\R^d$). Integrals on $\T$, $\Td$ or $\T^d$ mean integrals on $[0,1)$,
$[0,1)\times\R^{d-1}$ or~$[0,1)^d$.

We denote by $\LL^2(\T^d)$ the space of functions on $\T^d$ whose square
is integrable on $\T^d$, and by $\HH^1(\T^d)$ the space of functions in
$\LL^2(\T^d)$ whose weak gradient is square integrable on $\T^d$. We use
similar notations on $\Td$ and $\T.$

For two functions $f$ and $g$ defined on
$\Td$ or $\T^d$, we denote $f*g$ the convolution {\it with respect to
  the first coordinate}, that is,
$$f*g(x)=\int_\T f(x^1-y^1,x\z)g(y^1,x\z)\dx y^1.$$
If $f$ is defined on $\T$, we also use the  notation $f*g$ to denote $$f*g(x)=\int_\T f(x^1-y^1)g(y^1,x\z)\dx y^1.$$
When $f$ and $g$ are defined on $\mathcal D=\Td$ or $\T^d$, the convolution in {\it all} the coordinates is denoted $f\star g$:
$$f\star g(x)=\int_{\mathcal D} f(x^1-y^1,x\z-y\z)g(y^1,y\z)\dx
y^1\dx y\z.$$

In the following, we call ``probability measure on $\T$'' (resp. on $\Td$,
$\T^d$) a nonnegative $\mathbb Z$-periodical (resp. $\mathbb Z$-periodical with respect to the
first coordinate, $\Z^d$-periodical) measure $\mu$ such that
$\mu([0,1))=1$ (resp $\mu([0,1)\times\R^{d-1})=1$, $\mu([0,1)^d)=1$).

When $\{X\}$ is a random variable taking values in $\T$ (resp. in $\Td$,
$\T^d$), we call
``distribution of $\{X\}$'' or ``law of $\{X\}$'' the probability measure $\mu$ on $\T$ (resp. on $\Td$,
$\T^d$) such that $$\E[f(\{X\})]=\int f(x)\mu(\dx x).$$

For a given probability measure $\mu$ on $\Td$ (resp. a probability density $u$) and
a given bounded function~$g$, we denote $\mu^g$ (resp. $u^g(x^1)\dx x^1$) the
marginal on $\T$ of the measure $g.\mu$ (resp. $g(x)u(x)\dx x$). Namely:
$$\mu^g(A)=\int_{A\times\R^{d-1}}g\dx\mu$$
and $$u^g(x^1)=\int_{\R^{d-1}}g(x^1,x\z)u(x^1,x\z)\dx x\z.$$
In particular, $\mu^1$ is the first coordinate marginal of $\mu$.
When we do not specify the measure in an integral, it is the Lebesgue
measure.

We will need the weighted spaces $$\LL^p(w)=\left\{\psi\in\LL^p(\Td)\textrm{
  s.t. }\|\psi\|_{\LL^p(w)}\stackrel{\rm def}=\left(\int_\Td |\psi|^p
  w\right)^{1/p}<\infty\right\},$$ for $1\leq p<\infty$, and$$\HH^1(w)=\left\{\psi\in\HH^1(\Td)\textrm{
  s.t. }\|\psi\|_{\HH^1(w)}\stackrel{\rm def}=\left(\int_\Td \left(|\psi|^2+|\nabla\psi|^2\right)w\right)^{1/2}<\infty\right\}$$
 with $w(x)=(1+|x\z|^2)^\lambda$, for some $\lambda>(d-1)/2.$ Notice that $w$ does not depend on the first coordinate $x^1$, and that
 there is a positive constant $K$ such that
\begin{equation}\label{eq:domination_w}
\forall x\in\Td,~|\nabla w(x)|\leq2\lambda(1+|x\z|^2)^{\lambda-1}\sum_{i=2}^d|x^i|\leq Kw(x).
\end{equation}

We will use several times the following statement:

\begin{lem}\label{lem:d1V_continu}
For a bounded function $g$, and $u\in\LL^2(w)$ one has, for some
constant $K$,
$$\|u^g\|_{\LL^2(\T)}\leq K\|g\|_{\LL^\infty(\Td)}\|u\|_{\LL^2(w)}.$$
If moreover, $g$ has bounded derivatives and $u\in\HH^1(w)$, then
$$\|u^g\|_{\HH^1(\T)}\leq K\|g\|_{\mathbb W^{1,\infty}(\Td)}\|u\|_{\HH^1(w)}.$$
The same inequalities hold with the non weighted norms in the right-hand
side, for $u$
respectively in $\LL^2(\T^d)$ and $\HH^1(\T^d)$.
\end{lem}
\begin{proof}
Recall that we assumed $\lambda\geq\frac{d-1}2$, so that $\frac1w$ is
integrable on $\R^d$ : $\int_{\R^d}\frac1w\dx x<\infty$. Consequently, we have the estimation
\begin{align*}\|u^g\|_{\LL^2(\T)}^2&=\int_\T\left|\int_{\R^{d-1}}gu\right|^2\\
&\leq \|g\|_{\LL^\infty(\Td)}^2\int_\T\left(\int_{\R^{d-1}}|u|^2w\int_{\R^{d-1}}\frac1w\right)\\
&\leq K\|g\|_{\LL^\infty(\Td)}^2\|u\|_{\LL^2(w)}^2.
\end{align*}
The proof is similar in the space $\HH^1(w)$.
\end{proof}

In the following, $K$ will denote some positive constant, whose value
can change from line to line.

\section{Assumptions and statement of the main results}\label{enonce}

In this paper, we consider a particular case of Equation
\eqref{eq:EDS_ABF} to simplify the argumentation: we assume~$\beta=1$ (this can be realized by a
  change of variable), $\mathcal
D=\Td$ or $\mathcal D=\T^d$. We consider as reaction coordinate the
first coordinate function
$\xi:\mathcal D\rightarrow\R$ defined by~$\xi(x)=\xi(x^1,x^2,\hdots,
  x^d)=x^1.$ This should not change the theoretical results, but will
  simplify the proofs.
The definition \eqref{eq:def_F} of $F$
is then reduced to $$F=\partial_1V,$$where $V$ is defined on $\T^d$
or $\Td.$

The two settings $\mathcal D=\T^d$ and $\mathcal D=\Td$ will be
respectively called the {\it compact} and the {\it non-compact} case.
Our results hold in both settings, and the proofs are
mostly identical, with some slight additional difficulties in
the non compact case. Thus, in those situations, we only give the proofs in the non-compact case.

With those assumptions, Equation \eqref{eq:EDS_ABF} rewrites

\begin{equation}\label{eq:EDS}\dx X_t=\big(-\nabla
  V(X_t)+\E\left[\partial_1V(X_t)|\{X^1_t\}\right]e_1\big)\dx
  t+\sqrt2\dx W_t,
\end{equation} $e_1$ denoting the first vector in the canonical basis of $\R^d$.
We will call solution to Equation \eqref{eq:EDS} a process
$\{X_t\}$ where $X_t$ satisfies \eqref{eq:EDS}. The initial condition of
\eqref{eq:EDS} is a random variable denoted~$X_0$, and is supposed to be independent of the
Brownian motion~$W$. We denote by $P_0$ the law of~$\{X_0\},$ which is a probability measure on $\mathcal D.$

To ensure the integrability of $\partial_1V(X_t)$, we make the following assumption
:
\begin{hyp}\label{hyp:V} $V$ is a twice continuously
  differentiable function, which has bounded first and second order partial derivatives. 
\end{hyp}
Notice that Assumption \ref{hyp:V} yields boundedness of the drift coefficient
in \eqref{eq:EDS}.
In the compact case, assumption \ref{hyp:V} is satisfied as soon as $V$ is a twice
differentiable function.

We have to make some assumptions on the initial
condition $X_0$. What is needed to prove our results will depend on whether
we consider the compact or the non compact case.
In the compact case, we consider the following assumption:

\begin{hyp}\label{hyp:cond_init}
The probability measure $P_0$ has a density $p_0$ lying in $\LL^2(\T^d)$ and whose first coordinate marginal $p^1_0$ is bounded from below
by a positive constant. (Notice that $p^1_0$ is a probability density on~$\T$.)
\end{hyp}

In the non compact case, we will need a stronger assumption: we have to
control the decay of the initial condition at infinity, so we work in
the weighted space $\LL^2(w)$. We will use, in addition to Assumption \ref{hyp:cond_init}, the following
one:
\begin{hyp}\label{hyp:cond_init_nc}
The density $p_0$ of $P_0$ lies in both $\LL^1(w)$ and $\LL^2(w)$.
\end{hyp}
Notice that Assumptions \ref{hyp:cond_init_nc} implies that $\{X_0\}$ has
finite moments of order less than $2\lambda$, and that Assumption \ref{hyp:V} then yields a control on the corresponding moments of any solution to~\eqref{eq:EDS}, uniformly in $t\in\R$ :
\begin{lem}\label{lem:L1w}
Under Assumptions \ref{hyp:V} and \ref{hyp:cond_init_nc}, on any bounded
  time interval $[0,T],$ the moments of order less than $2\lambda$ of any
  solution $X$ of~\eqref{eq:EDS} are bounded:$$\sup_{0\leq t\leq T}\E[|X_t|^{2\lambda}]<\infty.$$
\end{lem}
\begin{proof}
This comes from the boundedness of the drift coeficient $b_s(x)=-\nabla V(x)+\E[\partial_1V(X)|X^1=x^1]$, which holds in
regard of Assumption~\ref{hyp:V}. Indeed, we have
$\E[|X_t|^{2\lambda}]=\E[|X_0+\int_0^tb_s(X_s)\dx s+\sqrt 2W_t|^{2\lambda}]\leq
K\left(\E[|X_0|^{2\lambda}]+t^{2\lambda}+t^\lambda\right)$, which is bounded on $[0,T]$.
\end{proof}

According to the following fundamental lemma, the solution to \eqref{eq:EDS} is going to sample
efficiently the coordinate reaction state space $\T.$
\begin{lem}\label{lem:chaleur}
Denote by $P_t$ the law of $\{X_t\}$, where $X_t$ is a solution to Equation~\eqref{eq:EDS}. Then,
$P_t^1$ has a density $p_t^1$, such that $p^1$ satisfies the heat equation on $\T$ with
initial condition $p^1_0.$ Thus,~$p^1$ is uniquely defined on
$\T\times[0,\infty)$, and smooth on $\T\times(0,\infty)$.
\end{lem}

\begin{proof}[Proof of Lemma~\ref{lem:chaleur}]
Let $f$ be a smooth function on $\T$. One has, by It\=o's formula
$$\partial_t\E\left[f(X_t^1)\right]=-\E\left[f'(X_t^1)\partial_1V(X_t)\right]+
\E\left[f'(X_t^1)\E\left[\partial_1V(X_t)|\{X_t^1\}\right]\right]+\E\left[f''(X_t^1)\right].$$
But, $f$ being a function on $\T$, $f'(X_t^1)$ only
depends on $\{X_t^1\}$, so that the two first terms in the right hand
side cancel. Then, it holds
$$\partial_t\E\left[f(X_t^1)\right]=\E\left[f''(X_t^1)\right],$$
which is exactly the heat equation in the weak sense for $t\mapsto p_t^1$,
$p_t^1$ being the
distribution of $\{X_t^1\}$.
For uniqueness and regularity of this solution, see \cite[Chapter XIV]{dautray-lions-99}.
\end{proof}

Lemma~\ref{lem:chaleur} allows us to rewrite equation \eqref{eq:EDS} using
 the distribution of $\{X_t^1\}$. Indeed, since $P_t^1$ has a density, the
 measure given for $A\subset[0,1)$ by $P_t^{\partial_1V}(A)=\E\left[\partial_1V(X_t)\mathbf 1_A(\{X_t^1\})\right]$
 also has a density $p_t^{\partial_1V}.$ We can thus write
\begin{align}\label{eq:EDS_densite}
\begin{cases}
\dx X_t&=\left(-\nabla
V(X_t)+\frac{p_t^{\partial_1V}(X^1_t)}{p_t^1(X^1_t)}e_1\right)\dx t+\sqrt2\dx W_t,\\
P_t&=\text{distribution of }\{X_t\}.
\end{cases}
\end{align}
Moreover, under Assumption \ref{hyp:cond_init} the density $p_t^1$ satisfies $0<\inf_\T p^1_0\leq p_t^1$,
uniformly in time, thanks to the maximum principle. This assumption will
consequently prevent the denominator in the second term of \eqref{eq:EDS_densite} from vanishing.

In view of Equation \eqref{eq:EDS_densite}, a natural particle
approximation of $X_t$ is then obtained using the Nadaraya-Watson
estimator of a conditional expectation (see
\cite{tsybakov-04}), given, for some parameter
$\eta$ and for a positive integer $N$, by the system of $N$
stochastic differential equations
\begin{equation}\label{eq:EDS_N_eps}\dx X_{t,n,N}^\eta=\left(-\nabla
  V(X_{t,n,N}^\eta)+\frac{\sum_{m=1}^N\phieps(X^{\eta,1}_{t,n,N}-X^{\eta,1}_{t,m,N})\partial_1V(X_{t,m,N}^\eta)}
{\sum_{m=1}^N\phieps(X^{\eta,1}_{t,n,N}-X^{\eta,1}_{t,m,N})}e_1\right)\dx t+\sqrt2\dx
W_t^n,~1\leq n\leq N
\end{equation}
where $(W^n_t)$ is a sequence of independent Brownian motions, and
$\phieps$ is a smooth approximation for the Dirac measure at the origin on $\T$.
For the initial condition, we work with the following assumption
\begin{hyp}\label{hyp:cond_init_part}
The initial condition of Equation
\eqref{eq:EDS_N_eps} is $(X_{0,n,N}^\eta)_{0\leq n\leq N}=(X_{0,n})_{0\leq
  n\leq N}$, where $(X_{0,n})_{n\in\mathbb N}$ is a sequence of i.i.d random
variables with density $p_0$, and independent of the Brownian motions $(W_t^n)_{t\geq0}$.
\end{hyp}

We also need an assumption on the shape of $\phieps$. The parameter
$\eta=(\alpha,\varepsilon)$ will be chosen in~$(0,\infty)^2$, and $\phieps$
will have the form 
\begin{equation}\label{eq:tete_phi}
\phieps(x)=\alpha+\psi_\varepsilon(x),
\end{equation}
where $\psi_\varepsilon$ is a sequence of mollifiers on $\T$ as
$\varepsilon\rightarrow0$. Namely, assuming $\varepsilon<1/2$, $\psi_\varepsilon$ is a smooth
non-negative $\Z$-periodical function, such that $\psi_\varepsilon\equiv0$
on $[-1/2,1/2]\setminus[-\varepsilon,\varepsilon]$ and such that
$$\int_{-1/2}^{1/2}\psi_\varepsilon=1.$$
A simple way to construct such a sequence is to consider a
smooth non-negative function $\psi$ defined on~$\R$, with support in $[-1,1]$ such that
$\int_\R\psi=1,$ and then consider the $\Z$-periodization
$\psi_\varepsilon$ of $\psi_\varepsilon=\frac1\varepsilon\psi(\frac.\varepsilon)$ ($\psi_\varepsilon$
is well
defined for $\varepsilon<1/2$).
This example makes the following assumption natural:

\begin{hyp}\label{hyp:phi}The function $\psi_\varepsilon$ satisfies
  $$\|\psi_\varepsilon\|_{\LL^\infty(\T)}\leq\frac K\varepsilon,\textrm{
    and }\|\psi_\varepsilon'\|_{\LL^\infty(\T)}\leq\frac K{\varepsilon^2}.$$
\end{hyp}
The reason for adding a positive
constant $\alpha$ to the mollifier is to avoid singularities at the
denominator in the right-hand side of \eqref{eq:EDS_N_eps}.
Notice that~\eqref{eq:tete_phi} yields strong existence and
uniqueness for~\eqref{eq:EDS_N_eps}, since the drift is globally Lipschitz continuous.

We are going to prove the following two results:

\begin{theo}\label{th:existence_loi}[Existence and uniqueness of the
  solution]
In both the compact and non compact cases, under Assumption \ref{hyp:V}, weak existence holds for Equation
\eqref{eq:EDS}. If $P$ denotes the
distribution of a solution, then for all $s>0$ the time marginals $P_s$ of $P$ admits
a density $p_s$, such that for all $0<t<T$,
\begin{equation}\label{eq:espace}p\in\LL^\infty((t,T),\LL^2(\mathcal
  D))\bigcap\LL^2((t,T),\HH^1(\mathcal D)).\end{equation}
Moreover, under both Assumptions \ref{hyp:V} and
\ref{hyp:cond_init} for the compact case, and under Assumptions
\ref{hyp:V}, \ref{hyp:cond_init} and~\ref{hyp:cond_init_nc} for the non
compact case, strong existence, pathwise uniqueness and uniqueness in distribution also hold, and one
can take $t=0$ in \eqref{eq:espace}.
\end{theo}

\begin{theo}\label{th:approx_particule}[Particle approximation of
  the process $X_t$]
In the compact case, under
Assumptions~\ref{hyp:V},~\ref{hyp:cond_init},~\ref{hyp:cond_init_part},
and~\ref{hyp:phi} or in the non-compact case under
the additional Assumption~\ref{hyp:cond_init_nc}, define the
processes $X_{t,n,N}$ by \eqref{eq:EDS_N_eps}.
Then, it holds that, for any positive $T$, and for $\alpha$ and
$\varepsilon$ small enough,
$$\E\left[\int_0^T\left\|\frac{\sum_{n=1}^N\partial_1V(X_{t,n,N}^\eta)
\phieps(.-X_{t,n,N}^{\eta,1})}{\sum_{n=1}^N\phieps(.-X_{t,n,N}^{\eta,1})}-A'_t\right\|_{\LL^\infty(\T)}\dx
t\right]=\mathcal
O\left(\alpha+\sqrt\varepsilon+\frac1{\sqrt N}e^{\frac
  K{\alpha\varepsilon^2}}\right).$$
\end{theo}

Theorem \ref{th:existence_loi} is a consequence of Theorem
\ref{th:unicite_EDS} and Corollary \ref{th:existence_EDS} below, and Theorem
\ref{th:approx_particule} is a consequence of Theorems
\ref{th:vitesse} and \ref{th:conv_N} below.

The convergence rate in Theorem \ref{th:approx_particule} is a pretty
bad one, since for a given $N$ it explodes as~$\varepsilon$ goes to
$0$. Consequently, the size $\varepsilon$ of the window has to be chosen carefully
depending on the number $N$ of particles. This is discussed more
precisely in Section \ref{numerique}.

\section{Notion of solution, regularity and uniqueness results}\label{regularite_unicite}

In this section we consider the Fokker-Planck equation associated to the nonlinear
stochastic differential equation~\eqref{eq:EDS} and prove that
uniqueness holds for weak solutions of this partial differential
equation. From this uniqueness result, the study of Equation
\eqref{eq:EDS} can be reduced to the study of a linear stochastic
differential equation. We can thus prove uniqueness for Equation \eqref{eq:EDS}.

Let us derive the Fokker-Planck equation associated to Equation \eqref{eq:EDS}.
Let~$\psi$ be a twice continuously differentiable function. Applying
It\=o's formula and taking the expectation, we obtain that the law $P_t$ of a weak solution $\{X_t\}$ to equation
\eqref{eq:EDS} satisfies
\begin{align}\label{eq:EDP_faible}
\int_\mathcal D\psi(x)\dx P_T(x)&=\int_\mathcal D\psi(x)\dx P_0(x)-\int_0^T\int_\mathcal D\nabla\psi(x)\cdot\nabla V(x)
\dx P_t(x)\dx t+\int_0^T\int_\mathcal D\Delta\psi(x)\dx P_t(x)\dx t\\
&\quad+\int_0^T\int_\mathcal D\partial_1\psi(x)\left(\frac{p_t^{\partial_1V}}{p_t^1}(x^1)\right)\dx
P_t(x)\dx t,\nonumber
\end{align}
which is a weak formulation of the following partial
differential equation
\begin{equation}\label{eq:EDP}\partial_tP_t=\text{div}\left(P_t\nabla V+\nabla
P_t\right)-\partial_1\left(P_t\frac{p_t^{\partial_1V}}{p_t^1}\right),
\end{equation}
with initial condition $P_0.$
Using integration by parts, we introduce a stronger definition for solutions to \eqref{eq:EDP}
which will allow us to prove existence and uniqueness.
\begin{defi}\label{def:sol_EDP}
In the compact case, a function $u$ is said to be a solution to \eqref{eq:EDP} if, for any
positive~$T$,
\begin{itemize}
\item $u$ belongs to $\LL^\infty((0,T),\LL^2(\T^d))\bigcap\LL^2((0,T),\HH^1(\T^d))$ ;
\item for any function $\psi\in\HH^1(\T^d)$, we have:
\begin{equation}\label{eq:EDP_dist}
\partial_t\int_\mathcal D u_t\psi=-\int_\mathcal D u_t\nabla
V\cdot\nabla\psi-\int_\mathcal D\nabla
u_t\cdot\nabla\psi+\int_\mathcal D
u_t\frac{u_t^{\partial_1V}}{u_t^1}\partial_1\psi,
\end{equation}
in the sense of distributions in time ;
\item $u_0=p_0$.
\end{itemize}
In the non compact case, $u$ is said to be a solution to \eqref{eq:EDP},
if, for any positive~$T$,
\begin{itemize}
\item $u$ belongs to
  $\LL^\infty((0,T),\LL^2(w))\bigcap\LL^2((0,T),\HH^1(w))$ ;
\item for any $\psi\in\HH^1(w)$ 
  \begin{equation}\label{eq:EDP_dist_nc}
\partial_t\int_\mathcal D u_t\psi w=-\int_\mathcal D u_t\nabla
V\cdot\left(w\nabla\psi+\psi\nabla w\right)-\int_\mathcal D\nabla
u_t\cdot\left(w\nabla \psi+\psi\nabla
w\right)+\int_\mathcal D
u_t\frac{u_t^{\partial_1V}}{u_t^1}(\partial_1\psi)w,
\end{equation} holds in the sense of distributions
  in time ;
\item $u_0=p_0$.
\end{itemize}
\end{defi}

Notice that \eqref{eq:EDP_dist} is a variational formulation of \eqref{eq:EDP} in the space $\LL^2(\T^d)$
and that \eqref{eq:EDP_dist_nc} is a variational formulation of \eqref{eq:EDP} in the space $\LL^2(w)$.

These conditions make sense. Indeed, in both cases, the conditions on
$u$ and $\psi$ are such that the variational formulations
\eqref{eq:EDP_dist} and \eqref{eq:EDP_dist_nc} are well defined (notice
that one has $|\nabla w|\leq Kw$).
Moreover, for the compact case, if $u$ lies in $\LL^2((0,T),\HH^1(\T^d))$, and
satisfies \eqref{eq:EDP_dist} then $\partial_tu$ lies in $\LL^2((0,T),\HH^{-1}(\T^d))$, so that (see
\cite[page 23]{lions-magenes-68}) $u$ lies in $\mathcal C([0,T],\LL^2(\T^d))$, allowing us to define the value of
$u$ at time $t=0$. The same argument holds for the non compact case.

\subsection{Existence of regular densities for solutions to the
  nonlinear equation}\label{sec:lien_EDS_EDP}

In this section, we consider a solution $X$ to
Equation~\eqref{eq:EDS} and we denote by $P_t$ the law of~$\{X_t\}$. We
show that $P_t$ has a density $p_t$, and that $p$ is a solution to
Equation~\eqref{eq:EDP}, in the sense of Definition~\ref{def:sol_EDP}.

\begin{lem}\label{lem:mild}
Consider both the compact and the non compact cases. Under Assumption
\ref{hyp:V}, for any~$t\geq0$, $P_t$ admits a density
$p_t$ with respect to the Lebesgue measure satisfying the following mild
representation
\begin{equation}\label{eq:mild}
p_t=G_t\star P_0+\int_0^t\nabla G_{t-s}\star(\nabla Vp_s)\dx
s-\int_0^t\partial_1G_{t-s}\star\left(\frac{p_s^{\partial_1V}}{p_s^1}p_s\right)\dx
s,
\end{equation}
 where $G_t$ is the density of $\sqrt2$ times the Brownian motion on
$\mathcal D$, namely $$G_t(x)=\frac1{(4\pi t)^{d/2}}\sum_{k\in\mathbb
  Z}e^{-\frac{|x-ke_1|^2}{4t}}$$ for the non-compact case, and $$G_t(x)=\frac1{(4\pi t)^{d/2}}\sum_{k\in\mathbb
  Z^d}e^{-\frac{|x-k|^2}{4t}}$$ for the compact case.
\end{lem}
\begin{proof}

Let $\chi$ be a smooth function with compact support on $\Td$ and $T>0.$
Then, for $t\in[0,T]$, the
function $\psi$ defined by
$$\psi_s=G_{t-s}\star\chi,$$ is the unique smooth
solution to the following problem
\begin{align}\label{eq:chaleur_retr}
\begin{cases}
\partial_s\psi&=-\Delta\psi\text{ on }(0,t)\times\Td,\\
\psi_t&=\chi\phantom{-\Delta}\text{ on }\Td.
\end{cases}
\end{align}
Computing $\psi_s(X_s)$ by It\=o's formula and using
\eqref{eq:chaleur_retr} we get 
\begin{align*}
\int_\Td\psi_t\dx P_t&=\int_\Td\psi_0\dx
P_0-\int_0^t\int_\Td\Delta\psi_s\dx P_s\dx
s+\int_0^t\int_\Td\Delta\psi_s\dx P_s\dx s\\
&\quad-\int_0^t\int_\Td\nabla\psi_s\cdot\nabla
V\dx P_s\dx
s+\int_0^t\int_\Td\partial_1\psi_s\frac{p_s^{\partial_1V}}{p_s^1}\dx P_s\dx s\\
&=\int_\Td\psi_0\dx P_0-\int_0^t\int_\Td\nabla\psi_s\cdot\nabla V\dx
P_s\dx
s+\int_0^t\int_\Td\partial_1\psi_s\frac{p_s^{\partial_1V}}{p_s^1}\dx P_s\dx s.
\end{align*}
Using the expression of $\psi_t$ and Fubini's Theorem, we have:
\begin{align*}
\int_\Td\chi \dx P_t&=\int_\Td\chi(G_t\star P_0)+\int_\Td\chi\int_0^t\nabla
G_{t-s}\star(P_s\nabla V)\dx s\\
&\quad-\int_\Td\chi\int_0^t\partial_1G_{t-s}\star\left(\frac{p_s^{\partial_1V}}{p_s^1}P_s\right)\dx s.
\end{align*}
This last equation being true for any smooth function $\chi$ with
compact support, then $P_t$ is given by the right-hand side of
\eqref{eq:mild}, which is an integrable function, so that for any
positive $t$, $P_t$ has a density $p_t$ satisfying \eqref{eq:mild}.
\end{proof}

In regard of the following lemma, $p$ necessarily satisfies
some integrability conditions.
\begin{lem}\label{lem:p_L2}
In both the compact and the non compact case, under Assumptions \ref{hyp:V} and
\ref{hyp:cond_init}, $p$ lies in $\LL^\infty((0,T),\LL^2(\mathcal D))$ for any $T>0$, and we have
$\|p\|_{\LL^\infty((0,T),\LL^2(\mathcal D))}\leq C$, where $C$ is some
constant only depending on $P_0$,
$\nabla V$ and $T$.

In the non compact case, under Assumptions \ref{hyp:V}, \ref{hyp:cond_init} and
\ref{hyp:cond_init_nc}, $p$ lies in $\LL^\infty((0,T),\LL^2(w))$ for any $T>0$, and we have a bound
$\|p\|_{\LL^\infty((0,T),\LL^2(w))}\leq C$, where $C$ is some
constant only depending on $P_0$,
$\nabla V$ and $T$.
\end{lem}
We only give the proof of Lemma
\ref{lem:p_L2} in the non compact case, the one in the compact case being similar.

\begin{proof}
The mild formulation \eqref{eq:mild} will allow us to prove that
$u\in\LL^\infty((0,T),\LL^2(w))$.
Since $p_0$ lies in both $\LL^1(w)$ and $\LL^2(w)$, it
lies in $\LL^q(w),$ for any $1\leq q\leq2.$ We first prove that we have
a uniform in time estimate in $\LL^q(w)$, $1\leq q\leq2$, for $p_t$.

 From equation
\eqref{eq:mild}, it follows
\begin{equation}\label{eq:mild_Lpw}\|p_t\|_{\LL^q(w)}\leq\|p_0\|_{\LL^q(w)}+\int_0^t\|\nabla
G_{t-s}\star(\nabla Vp_s)\|_{\LL^q(w)}+\left\|\partial_1G_{t-s}\star\left(\frac{p_s^{\partial_1V}}{p_s^1}p_s\right)\right\|_{\LL^q(w)}\dx s.
\end{equation}
It holds, from Jensen's inequality,

\begin{align*}
\|\nabla G_{t-s}\star(\nabla Vp_s)\|_{\LL^q(w)}^q\leq&K\int_\Td\left(|\nabla
G_{t-s}|\star p_s\right)^qw\\
\leq&K\int_\Td\left(|\nabla
G_{t-s}|^q\star p_s\right)w\\
=&K\int_\Td\int_\Td|\nabla G_{t-s}(y)|^qp_s(x-y)w(x)\dx x\dx y.
\end{align*}
Now, notice that $w(x)\leq K(1+|y\z|^{2\lambda})w(x-y)\stackrel{\rm def}=\pi(y)w(x-y)$, so that 
\begin{align*}
\|\nabla G_{t-s}\star(\nabla Vp_s)\|_{\LL^q(w)}^q\leq&K\int_\Td\int_\Td|\nabla
G_{t-s}(y)|^q\pi(y)p_s(x-y)w(x-y)\dx x\dx y\\
=&K\|p_s\|_{\LL^1(w)}\int_\Td|\nabla G_{t-s}(y)|^q\pi(y)\dx y.
\end{align*}
In view of Lemma \ref{lem:L1w}, $\|p_s\|_{\LL^1(w)}$ is
bounded. Moreover, one has for $0\leq s\leq t\leq T$,
\begin{align*}
|\nabla
G_{t-s}(y)|^q\pi(y)&=\left|\left(4\pi(t-s)\right)^{-d/2}\sum_{k\in\Z}
-\frac{y-ke_1}{2(t-s)}e^{-\frac{|y-ke_1|^2}{4(t-s)}}\right|^q(1+|y\z|^{2\lambda})\\
&\leq \frac K{(t-s)^{q(d+1)/2}}\left|\left(1+\frac{|y\z|^{2\lambda/q}}{(t-s)^{\lambda/q}}\right)\sum_{k\in\Z}\frac{|y-ke_1|}{\sqrt{t-s}}e^{-\frac{|y-ke_1|^2}{4(t-s)}}\right|^q.\\
\end{align*}
Then, since a function $f$ with polynomial growth satisfies $f(x)e^{-x^2}\leq Ke^{-x^2/2}$ for
some constant~$K$, using H\"older's inequality, we deduce,
\begin{align*}
|\nabla
G_{t-s}(y)|^q\pi(y)&\leq\frac K{(t-s)^{q(d+1)/2}}\left|\sum_{k\in\Z}e^{-\frac{|y-ke_1|^2}{8(t-s)}}\right|^q\\
&\leq\frac
K{(t-s)^{q(d+1)/2}}\left|\sum_{k\in\Z}e^{-\frac{q|y-ke_1|^2}{16(t-s)}}\right|\left|\sum_{k\in\Z}e^{-\frac{{q'}|y-ke_1|^2}{16T}}\right|^{\frac
q{q'}}\\
&\leq\frac
K{(t-s)^{q(d+1)/2}}\sum_{k\in\Z}e^{\frac{-q|y-ke_1|^2}{16(t-s)}},
\end{align*}
where $q'$ satisfies $\frac1q+\frac1{q'}=1$.
Consequently, it holds, for $0\leq s\leq t\leq T,$

\begin{equation}\label{eq:gradG_Lq}\left(\int_\Td|\nabla G_{t-s}(y)|^q\pi(y)\dx y\right)^{1/q}\leq\frac K{(t-s)^{(d+1)/2-d/2q}}.\end{equation}
The last term in \eqref{eq:mild_Lpw} can be bounded in the same way, so
we deduce that $\int_0^t\|\nabla
G_{t-s}\star(\nabla Vp_s)\|_{\LL^q(w)}+\left\|\partial_1G_{t-s}\star\left(\frac{p_s^{\partial_1V}}{p_s^1}p_s\right)\right\|_{\LL^q(w)}\dx s$ is finite as soon as
\begin{equation}\label{eq:condition_qw}1\leq
  q<\frac d{d-1}.\end{equation}
In view of \eqref{eq:mild_Lpw}, $p$ lies in
$\LL^\infty((0,T),\LL^q(\Td))$ for all $T$ and all $q$ satisfying
\eqref{eq:condition_qw}, and we have a bound on its norm depending only
on $P_0,$ $\nabla V$ and $T$. We now bootstrap this
estimate to reach a uniform-in-time
$\LL^2(w)$ bound for $p$.

Let $n_0$ be an integer large enough so that
$$\frac{n_0+1}{n_0+1/2}<\frac d{d-1},$$ and define for
$n=0,\hdots n_0$, $q=\frac{n_0+1}{n_0+1/2}$ and $q_n=\left(\frac1q+n\left(\frac1q-1\right)\right)^{-1}.$
Notice that $(q_n)_{n=0\hdots n_0}$ satisfies $q_0=q$, $q_{n_0}=2$ and $$1+\frac1{q_{n+1}}=\frac1{q_n}+\frac1q,$$
so that, according to Young's Inequality, convolution continuously
maps $\LL^{q_n}\times\LL^q$ to $\LL^{q_{n+1}}$.
Consequently, we have for $n<n_0$
\begin{align*}
\|\nabla G_{t-s}\star(\nabla
Vp_s)\|_{\LL^{q_{n+1}}(w)}\leq&K\left(\int_\Td(|\nabla G_{t-s}|\star
    p_s)^{q_{n+1}}(x)w(x)\dx x\right)^{1/{q_{n+1}}}\\
=&K\left(\int_\Td\left(\int_\Td|\nabla G_{t-s}|(y)p_s(x-y)\dx
    y\right)^{q_{n+1}}w(x)\dx x\right)^{1/q_{n+1}}.
\end{align*}
We have $w(x)\leq w(x-y)\pi(y)\leq w(x-y)^{q_{n+1}/q_n}\pi(y)$, since
$q_n\leq q_{n+1},$ yielding, by Young's inequality and the polynomial
growth of $\pi,$
\begin{align*}\|\nabla
G_{t-s}\star(\nabla
Vp_s)\|_{\LL^{q_{n+1}}(w)}\leq&K\left(\int_\Td\left(\int_\Td|\nabla G_{t-s}|(y)\pi(y)^{1/{q_{n+1}}}p_s(x-y)w(x-y)^{1/q_n}\dx
    y\right)^{q_{n+1}}\dx x\right)^{1/q_{n+1}}\\
=&K\|(|\nabla G_{t-s}|\pi^{1/q_{n+1}})\star(p_sw^{1/q_n})\|_{\LL^{q_{n+1}}(\Td)}\\
\leq&K\||\nabla
G_{t-s}|\pi^{1/q_{n+1}}\|_{\LL^q(\Td)}\|p_s\|_{\LL^{q_n}(w)}\\
\leq&\frac K{(t-s)^{(d+1)/2-d/(2q)}}\|p_s\|_{\LL^{q_n}(w)}.
\end{align*}
the last inequality being proved in the same way as \eqref{eq:gradG_Lq} is.
As a result, for $n<n_0$,
$$\|p_t\|_{\LL^{q_{n+1}}(w)}\leq\|p_0\|_{\LL^{q_{n+1}}(w)}+K\int_0^t\frac{\|p_s\|_{\LL^{q_n}(w)}}{(t-s)^{(d+1)/2-d/(2q)}}\dx s.$$
By induction on $n$, since $\frac 1{(t-s)^{(d+1)/2-d/(2q)}}$ is
integrable on $[0,t]$, this estimate shows that $p$ lies in
$\LL^\infty((0,T),\LL^2(w))$, for all positive $T$. Since we control
$\sup_{t\in[0,T]}\|p_t\|_{\LL^{q_0}(w)}$ by a constant depending only on $P_0,$ $\nabla V$
and $T$, we also have such a control on $\sup_{t\in[0,T]}\|p_t\|_{\LL^2(w)}$.

\end{proof}

Now, we prove that $p$ is a solution to Equation \eqref{eq:EDP}  in the
sense of Definition \ref{def:sol_EDP}. First, we show that it satisfies
the regularity condition.

\begin{lem}\label{lem:regularite_EDP}
In the compact case, under Assumptions
\ref{hyp:V} and~\ref{hyp:cond_init}, one has
\begin{equation}\label{eq:regularite}p\in\LL^\infty\left((0,T),\LL^2(\T^d)\right)\bigcap\LL^2\left((0,T),\HH^1(\T^d)\right).\end{equation}
Moreover $\|p\|_{\LL^\infty((0,T),\LL^2(\T^d))}+\|p\|_{\LL^2((0,T),\HH^1(\T^d))}\leq
K$, where $K$ only depends on $\nabla V$, $P_0$ and $T$.

In the non compact case, with the additional Assumption
\ref{hyp:cond_init_nc}, one has
\begin{equation}\label{eq:regularite_nc}p\in\LL^\infty\left((0,T),\LL^2(w)\right)\bigcap\LL^2\left((0,T),\HH^1(w)\right).\end{equation}
Moreover $\|p\|_{\LL^\infty((0,T),\LL^2(w))}+\|p\|_{\LL^2((0,T),\HH^1(w))}\leq
K$, where $K$ only depends on $\nabla V$, $P_0$ and $T$.
\end{lem}
\begin{proof}
According to Assumption \ref{hyp:cond_init}, $p_0$ lies in
$\LL^2(\mathcal D)$. Consequently, from Lemma \ref{lem:p_L2}, we know that~$P_t$ has a density $p_t$ such that
$p\in\LL^\infty((0,T),\LL^2(\mathcal D)).$
We now prove that $p$ lies in $\LL^2((0,T),\HH^1(\mathcal D)).$
We know that $p$ lies in $\LL^\infty((0,T),\LL^2(\mathcal
D))\subset\LL^2((0,T),\LL^2(\mathcal D))$, and that $\frac{p^{\partial_1V}}{p^1}$ is in $\LL^\infty([0,T]\times\mathcal D)$, so that the function $f$ defined by 
$$f=\text{div}(p\nabla
V)-\partial_1\left(\frac{p^{\partial_1V}}{p^1}p\right)$$
lies in $\LL^2((0,T),\HH^{-1}(\mathcal D)).$
Consequently, it can be shown, for example using a Galerkin approximation (see
\cite[Chapter XVIII]{dautray-lions-99}) that the problem
\begin{equation}\label{eq:EDP_chaleur}\left\{\begin{array}{rcl}
\partial_tv-\Delta v&=&f,\\
v_0&=&p_0,
\end{array}\right.\end{equation}
admits a unique weak solution $v$ in the space
$\LL^\infty((0,T),\LL^2(\mathcal D)\bigcap\LL^2((0,T),\HH^1(\mathcal D)).$ Here, ``weak solution'' means that for any $\psi$ in
$\HH^1(\mathcal D),$ 
\begin{equation}\label{eq:varia}\partial_t\int_{\mathcal D}\psi
  v_t+\int_{\mathcal D}\nabla\psi\nabla v_t=\int_{\mathcal D}\psi f\end{equation}
holds. Thanks to an {\it a priori} estimate, we can find a bound $K$ depending
only on $\nabla V$, $P_0$ and~$T$,
such that this weak solution lies in the ball of radius $C$ in the
spaces $\LL^\infty((0,T),\LL^2(\mathcal D))$ and
$\LL^2((0,T),\HH^1(\mathcal D))$.
For the non compact case, notice that under Assumption
\ref{hyp:cond_init_nc}, $f$ satisfies for any $\psi\in\HH^1(w),$
\begin{align*}
\left|\int_\Td f\psi w\right|=&\left|\int_\Td p\nabla
V\cdot\nabla(\psi w)-\frac{p^{\partial_1V}}{p^1}p\,\partial_1(\psi
w)\right|\\
\leq&K\int_\Td|p\nabla\psi| w+K\int_\Td|p\psi| w\\
\leq& K\|\psi\|_{\HH^1(w)},
\end{align*}
the last bound being deduced from Lemma \ref{lem:p_L2}.
From the following {\it a priori} estimate,
\begin{align*}
\frac12\partial_t\|v_t\|_{\LL^2(w)}^2=&-\int_\Td\nabla v_t
\nabla(wv_t)+\int_\Td fv_tw\\
\leq&-\int_\Td|\nabla v_t|^2w+K\int_\Td|v_t\nabla
v_t|w+K\|v_t\|_{\HH^1(w)}\\
\leq&-\frac12\|\nabla v_t\|_{\LL^2(w)}^2+K\|v_t\|_{\LL^2(w)}^2+K,
\end{align*}
standard arguments show that $v$ also lies in
$\LL^\infty((0,T),\LL^2(w))\bigcap\LL^2((0,T),\HH^1(w)),$ if $p_0\in\LL^2(w).$

We are now going to show that $v$ is actually equal to the function $p$.
For a fixed $t$ in $[0,T]$, consider $\psi_s=G_{t-s}\star\chi$, solution to the problem
\eqref{eq:chaleur_retr}, where $\chi$ is some test function, and compute
$\partial_s\int_\mathcal D\psi_sv_s$. From \cite[page 261, Lemma 1.2]{temam-79},
we obtain
$$\partial_s\int_{\mathcal D}\psi_sv_s=\int_{\mathcal D}\psi_sf,$$in the sense of distributions.
Using the expression of $\psi_s$, this equation rewrites
$$\partial_s\int_\mathcal D(G_{t-s}\star\chi)v_s=\int_\mathcal D(G_{t-s}\star\chi)f,$$
which is equivalent to 
\begin{equation}\label{eq:identification}\partial_s\int_\mathcal D\chi(G_{t-s}\star v_s)=\int_\mathcal D\chi(G_{t-s}\star
f).\end{equation}
Since $v\in\LL^2((0,T),\HH^1(\mathcal D)),$ and
$\partial_sv\in\LL^2((0,T),\HH^{-1}(\mathcal D))$, then $v$ lies in
$\mathcal C((0,T),\LL^2(\mathcal D))$ (see \cite[Chapter XVIII, \S1, Theorem 1]{dautray-lions-99}, so that the left hand side in
\eqref{eq:identification} is the derivative with respect ot $s$ of a function which is
continuous in $s$. Moreover, one has $$G_{t-s}\star f=\nabla
G_{t-s}\star(p\nabla
V)-\partial_1G_{t-s}\star\left(\frac{p^{\partial_1V}}{p^1}p\right)\in\LL^1((0,t),\LL^2(\mathcal
D)),$$ so that the right hand side in \eqref{eq:identification} is
integrable in time.
Consequently, integrating on $[0,t]$, one finds
\begin{equation*}\int_\mathcal D\chi v_t=\int_\mathcal
  D\chi(G_t\star p_0)+\int_\mathcal D\int_0^t\chi(\nabla G_{t-s}\star(\nabla
Vp_s))\dx
s-\int_\mathcal D\int_0^t\chi\left(\partial_1G_{t-s}\star\left(\frac{p_s^{\partial_1V}}{p_s^1}p_s\right)\right)\dx
s.\end{equation*}
Identifying in the sense of distribution, one has
\begin{equation}\label{eq:mild2}v_t=G_t\star p_0+\int_0^t\nabla G_{t-s}\star(\nabla
Vp_s)\dx
s-\int_0^t\partial_1G_{t-s}\star\left(\frac{p_s^{\partial_1V}}{p_s^1}p_s\right)\dx
s.\end{equation}
The right hand side in \eqref{eq:mild2} is exactly the right hand side
in \eqref{eq:mild}, and \eqref{eq:mild2} holds for all $t>0$, so that
$v=p$, and the regularity we wanted on $p$ actually holds.

\end{proof}

We finish this section by proving:
\begin{lem}\label{lem:lien_EDS_EDP}
The function $p$ satisfies Equation
\eqref{eq:EDP} in the sense of Definition \ref{def:sol_EDP}.
\end{lem}
\begin{proof}
According to Lemma~\ref{lem:regularite_EDP}, in the compact case
(resp. in the non compact case), for any $t>0$, $p$ lies in $\LL^\infty((0,T),\LL^2(\T^d)\bigcap\LL^2((0,T),\HH^1(\T^d))$ (resp. in
$\LL^\infty((0,T),\LL^2(w)\bigcap\LL^2((0,T),\HH^1(w))$).
 Moreover, thanks to It\=o's Formula, $p$
satisfies Equation~\eqref{eq:EDP_faible} for any smooth test function
$\psi$. But, according to the regularity of~$p_t$, and by the density of
smooth functions in $\HH^1(\T^d)$ (resp. in $\HH^1(w)$), Equation~\eqref{eq:EDP_dist}
holds for any $\psi$ in $\HH^1(\T^d)$ (resp. \eqref{eq:EDP_dist_nc} holds
for any $\psi$ in $\HH^1(w)$). This means
that $p_t$ is a solution to \eqref{eq:EDP} in the sense of Definition \ref{def:sol_EDP}. 
\end{proof}

\subsection{Uniqueness results}
In this section we prove that uniqueness holds for solutions of Equation
\eqref{eq:EDP} in the sense of Definition~\ref{def:sol_EDP}, yielding
uniqueness for solutions of the nonlinear equation \eqref{eq:EDS}.
\subsubsection{Uniqueness for the Fokker-Planck Equation}

\begin{theo}\label{th:unicite_EDP}In the compact case, under Assumptions
  \ref{hyp:V} and \ref{hyp:cond_init} or in the non compact case under
  Assumptions \ref{hyp:V}, \ref{hyp:cond_init} and \ref{hyp:cond_init_nc}, uniqueness holds for the
  solutions to the Fokker-Planck
  equation \eqref{eq:EDP} in the sense of
Definition~\ref{def:sol_EDP}.
\end{theo}

\begin{proof}
We only give the proof in the non compact case, which can be
adapted straightforwardly for the compact case by
performing the same computations in the space $\LL^2(\T^d)$.
Let $u$ and $v$ be two solutions of \eqref{eq:EDP} in the sense of
Definition \ref{def:sol_EDP} with same initial condition $u_0=v_0$. We use Gr\"onwall's
Lemma to prove that $\|u_t-v_t\|_{\LL^2(w)}=0$ for all $t>0$. 
Adapting the proof of \cite[page 261, Lemma 1.2]{temam-79}, one has
$\frac12\partial_t\|u_t-v_t\|^2_{\LL^2(w)}=\int_{\Td}(u_t-v_t)\partial_t(u_t-v_t)w$.
Consequently, since $u$ and $v$ satisfy Definition~\ref{def:sol_EDP}, and using \eqref{eq:domination_w} and
Assumption \ref{hyp:V}, it
holds that
\begin{align*}
\frac12\partial_t\|u_t-v_t\|_{\LL^2(w)}^2
&\leq K\|u_t-v_t\|_{\LL^2(w)}^2+K\|u_t-v_t\|_{\LL^2(w)}\|\nabla u_t-\nabla v_t\|_{\LL^2(w)}-\|\nabla
u_t-\nabla v_t\|^2_{\LL^2(w)}\\
&\quad+\int_{\Td}\partial_1(u_t-v_t)\left(u_t\frac{u_t^{\partial_1V}}{u_t^1}-v_t\frac{v_t^{\partial_1V}}{v_t^1}\right)w.
\end{align*}
We want to estimate the last term.
Notice that, thanks to Lemma~\ref{lem:chaleur}, $u^1=v^1$, so that 
\begin{align*}
\int_{\Td}\partial_1(u_t-v_t)\left(u_t\frac{u_t^{\partial_1V}}{u_t^1}-v_t\frac{v_t^{\partial_1V}}{v_t^1}\right)w
&=\int_{\Td}\partial_1(u_t-v_t)u_t\frac{u_t^{\partial_1V}-v_t^{\partial_1V}}{u_t^1}w\\
&\quad+\int_{\Td}\partial_1(u_t-v_t)(u_t-v_t)\frac{v_t^{\partial_1V}}{u_t^1}w.
\end{align*}
Since $\partial_1V$ is bounded, the second term in the right-hand side
is smaller than
$$K\|u_t-v_t\|_{\LL^2(w)}\|\nabla u_t-\nabla v_t\|_{\LL^2(w)},$$
and the first term is smaller than
$$\|\nabla
u_t-\nabla v_t\|_{\LL^2(w)}\left(\int_{\Td}
  \left(\frac{u_t}{u_t^1}\right)^2\left(u_t^{\partial_1V}-v_t^{\partial_1V}\right)^2w\right)^{1/2}.$$
Then,
\begin{align*}
\left(\int_{\Td}
  \left(\frac{u_t}{u_t^1}\right)^2\left(u_t^{\partial_1V}-v_t^{\partial_1V}\right)^2w\right)^{1/2}&=\left(\int_\T
  \left(\frac{u_t^{\partial_1V}-v_t^{\partial_1V}}{u_t^1}\right)^2\left(\int_{\R^{d-1}}(u_t)^2w\right)\right)^{1/2}\\
&\leq\left\|\frac1{u_t^1}(u_t^{\partial_1V}-v_t^{\partial_1V})\right\|_{\LL^\infty(\T)}\|u_t\|_{\LL^2(w)}.
\end{align*}
The function $t\mapsto\|u_t\|_{\LL^2(w)}$
is bounded on $[0,T]$, and, thanks to Lemma~\ref{lem:chaleur}, Assumption \ref{hyp:cond_init} and the maximum
principle, $u^1$ is bounded from below by some positive constant, so that
$$\left(\int_{\Td}\left(\frac{u_t}{u_t^1}\right)^2
\left(u_t^{\partial_1V}-v_t^{\partial_1V}\right)^2w\right)^{1/2}\leq K\|u_t^{\partial_1V}-v_t^{\partial_1V}\|_{\LL^\infty({\T})}.$$
To conclude, notice that, for any positive $\gamma$,
$\HH^{1/2+\gamma}(\T)$ continuously imbeds in $\mathcal
C(\T)$ (see \cite[page 217]{adams-78}). Consequently, interpolating $\HH^1(\T)$
and $\LL^2(\T)$ (see \cite[Page 49]{lions-magenes-68}, we obtain for a function $f$ in $\HH^1(\T)$ and $\gamma\in(0,\frac12)$,
\begin{equation}\label{eq:interpolation}
\|f\|_{\LL^\infty(\T)}\leq K\|f\|_{\HH^{1/2+\gamma}(\T)}\leq
K\|f\|_{\LL^2(\T)}^{1/2-\gamma}\|f\|_{\HH^1(\T)}^{1/2+\gamma}.
\end{equation}
All the previous inequalities give us
\begin{align*}
\frac12\partial_t\|u_t-v_t\|^2_{\LL^2(w)}+\|\nabla u_t-\nabla
v_t\|_{\LL^2(w)}^2&\leq K\|u_t-v_t\|_{\LL^2(w)}\|\nabla u_t-\nabla v_t\|_{\LL^2(w)}+K\|u_t-v_t\|_{\LL^2(w)}^2\\
&\quad+K\|u_t^{\partial_1V}-v_t^{\partial_1V}\|_{\LL^2(\T)}^{1/2-\gamma}\|u_t^{\partial_1V}-v_t^{\partial_1V}\|_{\HH^1(\T)}^{1/2+\gamma}\|\nabla
u_t-\nabla v_t\|_{\LL^2(w)}.\end{align*}
We finally obtain, from Lemma \ref{lem:d1V_continu} and Young's
inequality $ab\leq\varepsilon a^p+q^{-1}(p\varepsilon)^{-q/p}b^q$, holding
true for any positive $a$, $b$, $\varepsilon$,
$p$ and $q$ such that $\frac1p+\frac1q=1,$
$$\partial_t\|u_t-v_t\|^2_{\LL^2(w)}+\|\nabla u_t-\nabla v_t\|_{\LL^2(w)}^2\leq K\|u_t-v_t\|_{\LL^2(w)}^2,$$
yielding uniqueness through Gr\"onwall's lemma.
\end{proof}

\begin{rem}\label{rem:unicite_nc}
A more natural uniqueness proof can be performed, using an
entropy estimate. In particular, this proof does not require the
introduction of the weighted spaces. Unfortunately, it does not apply to
the solutions in the sense of Definition \ref{def:sol_EDP}. Uniqueness actually
holds in the subspace of functions such that the following computations
make sense.

 Let $u$ and $v$ be two
solutions of \eqref{eq:EDP} with same initial condition $u_0=v_0$. Notice that from Lemma~\ref{lem:chaleur}, the
functions $u^1$ and $v^1$ are equal.
Define the relative entropy of $u$ with respect to $v$:
$$E(t)=\int_\Td u\log\frac uv.$$
If all quantities involved are finite, it holds that

\begin{align*}
E'(t)&=\partial_t\left(\int_\Td u\right)+\int_\Td\partial_tu~\log\frac uv-\int_\Td \partial_t v~\frac uv\\
&=0-\int_\Td u\nabla V~\nabla\log\frac uv-\int_\Td\nabla u~\nabla\log\frac
uv+\int_\Td u\frac{u^{\partial_1V}}{u^1}\partial_1\log\frac uv\\
&\quad+\int_\Td v\nabla V~\nabla\frac uv+\int_\Td \nabla
v~\nabla\frac uv-\int_\Td v\frac {v^{\partial_1V}}{u^1}\partial_1\frac uv\\
&=-\int_\Td \frac{v^2}u\left|\nabla\frac
  uv\right|^2+\int_\Td\left(u^{\partial_1V}-v^{\partial_1V}\right)~\frac
v{u^1}\partial_1\frac uv.
\end{align*}
But, using Csisz\'ar-Kullback inequality, it holds
\begin{align*}
\int_\Td\left(u^{\partial_1V}-v^{\partial_1V}\right)~\frac
v{u^1}\partial_1\frac uv
&\leq K\int_\Td v\left|\partial_1\frac
  uv\right|\left\|\frac u{u^1}-\frac v{u^1}\right\|_{TV(\Rdd)}\\
&\leq K\int_\Td v\left|\partial_1\frac
  uv\right|\left(\int_\Rdd\left(\frac u{u^1}\log\frac uv\right)\right)^{1/2}.\\
\end{align*}
In conclusion, we find
\begin{align*}
E'(t)
&\leq-\int_\Td \frac{v^2}u\left|\nabla\frac
  uv\right|^2+K\left(\int_\Td\frac{v^2}u\left|\partial_1\frac
  uv\right|^2\right)^{1/2}\left(E(t)\right)^{1/2}.
\end{align*}
We can conclude the proof using Young's
inequality and then Gr\"onwall's
Lemma.
\end{rem}

\subsubsection{Uniqueness for the nonlinear process}

\begin{theo}\label{th:unicite_EDS}Pathwise uniqueness and uniqueness in
  law hold for Equation \eqref{eq:EDS} in the compact case under
  Assumptions \ref{hyp:V} and \ref{hyp:cond_init}, and in the non
  compact case under Assumptions \ref{hyp:V}, \ref{hyp:cond_init} and \ref{hyp:cond_init_nc}.
\end{theo}
\begin{proof}
As stated in Lemma \ref{lem:lien_EDS_EDP}, if $X$
solves \eqref{eq:EDS}, then $\{X_t\}$ admits
a density $p_t$ such that~$p$ satisfies \eqref{eq:EDP} in the sense of
Definition~\ref{def:sol_EDP}.
Thus, in regard of
Theorem~\ref{th:unicite_EDP}, $p_t$ is uniquely defined. Consequently, Equation~\eqref{eq:EDS} rewrites
\begin{equation}\label{eq:EDS2}\dx X_t=\left(-\nabla
    V(X_t)+\frac{p_t^{\partial_1V}(X_t^1)}{p_t^1(X_t^1)}e_1\right)\dx
  t+\sqrt2\dx W_t,\end{equation}
where $p_t$ is the unique solution to Equation \eqref{eq:EDP} in the
sense of Definition \ref{def:sol_EDP}.
Notice that the drift  
$$b_t(x)=-\nabla V(x)+\frac{p_t^{\partial_1V}(x^1)}{p_t^1(x^1)}e_1$$
in Equation \eqref{eq:EDS2} is bounded, so that pathwise uniqueness
holds (see \cite{krylov-rockner-05}), as well as uniqueness in law, from the Girsanov Theorem. 
\end{proof}

\section{A regularized approximate dynamics}\label{dynamique_approchee}

To estimate the difference between the nonlinear process defined
by Equation \eqref{eq:EDS} and its particle approximation
\eqref{eq:EDS_N_eps}, we introduce an intermediate process, called the
regularized nonlinear process, which is the natural expected limit as $N$ goes to infinity of
the particle approximation \eqref{eq:EDS_N_eps}. The nonlinear term in
this equation is more regular than the one in \eqref{eq:EDS}, so that we
can show existence and uniqueness for this process.

The aim of this section will be, in a first time, to prove existence and
uniqueness for the regularized nonlinear process, see Theorem~\ref{th:ex_un_pnl}, and
in a second time to show that the regularized nonlinear process converges to the
nonlinear process solution to \eqref{eq:EDS} as $\varepsilon$ and $\alpha$ go to zero, and to
estimate the rate of this convergence, see Theorem \ref{th:vitesse} below.
This will yield an existence result for the nonlinear process.

Under Assumption \ref{hyp:cond_init_part} on the initial condition, for a
fixed positive integer $n$, we expect the sequence of processes $(X_{n,N}^\eta)_{N>0}$
 defined by~\eqref{eq:EDS_N_eps} to converge to a solution to
\begin{align}\label{eq:EDS_eps}
\begin{cases}\dx\bar X_{t,n}^\eta&\displaystyle=\left(-\nabla
V(\bar X_{t,n}^\eta)+\frac{\phieps*P_t^{\eta,\partial_1V}}{\phieps*P_t^{\eta,1}}(\bar X_{t,n}^{\eta,1})e_1\right)\dx t+\sqrt2\dx W_t^n,\\
P_t^\eta&=\text{distribution of }\{\bar X_{t,n}^\eta\}
\end{cases}
\end{align}
with initial condition $(X_{0,n})$.

\subsection{Existence and uniqueness for the regularized problem}
In this section, we show that pathwise uniqueness, uniqueness in
distribution and strong existence hold for the regularized
dynamics. 

We first show existence and uniqueness of a solution to
\eqref{eq:EDS_eps}, using a fixed point method.

\begin{theo}\label{th:ex_un_pnl}
Consider both the compact and the non compact cases. Under Assumptions
\ref{hyp:V} and~\ref{hyp:cond_init_part}, strong
existence and uniqueness hold for Equation \eqref{eq:EDS_eps}.
\end{theo}
Here we follow \cite{sznitman-89}: we show that a measure on the space of
continuous paths from $[0,T]$ to $\R^d$ is the law of a
solution to \eqref{eq:EDS_eps} if and only if it is a fixed point of
some function $\Phi_T$. Then we show existence and uniqueness of this fixed
point by a contraction argument. This cannot be done directly for
Equation \eqref{eq:EDS}, since its nonlinear term is too ill-behaved, so
that we do not have contraction in that case.

For a probability measure $\mu$ on the set $\mathcal C_T=\mathcal
C([0,T],\R^d)$ we denote by $\Phi_T(\mu)$ the distribution on $\C$ of the process $X$ defined by
\begin{equation}\label{eq:def_Phi}\dx X_t=\left(-\nabla
V(X_t)+\frac{\displaystyle\int_\C\phieps(X^1_t-x^1_t)\partial_1V(x_t)\dx
 \mu(x)}{\displaystyle\int_\C \phieps(X^1_t-x^1_t)\dx
 \mu(x)}e_1\right)\dx t+\sqrt2\dx W_t\end{equation}
whose initial condition $X_0$ has law $P_0$ and is independent of $W$. The distribution $\Phi_T(\mu)$ is well defined since Equation
\eqref{eq:def_Phi}, having global Lipschitz
coefficients, has a unique strong solution.

Notice that, since 
$$\displaystyle\frac{\phieps*\mu_t^{\partial_1V}}{\phieps*\mu_t^1}=\frac{\int_\C\phieps(\cdot-x^1_t)\partial_1V(x_t)\dx
  \mu(x)}{\int_\C \phieps(\cdot-x^1_t)\dx \mu(x)},$$
 $\mu$ is the distribution of a solution to
\eqref{eq:EDS_eps} up to time $T$ if, and
only if $\Phi_T(\mu)=\mu$. We will show that such a $\mu$ exists and is
unique using Picard's Theorem.

The Wasserstein metric $D_T(\mu_1,\mu_2)$ between two probability distributions $\mu_1$
and $\mu_2$ on~$\mathcal C_T$ is defined by $$D_T(\mu_1,\mu_2)=\inf_{\pi\in\Pi}\int_{\C\times\C} 1\wedge\|x-y\|_{\mathcal C_T}\dx
  \pi(x,y),$$ 
where $\Pi=\left\{\pi\in\mathcal P(\mathcal C_T\times\mathcal
C_T),~\pi\textit{ having }\mu_1\textit{ and }\mu_2\textit{ as marginal
  distributions}\right\}$ is the set of all coupling of $\mu_1$ and $\mu_2$,
and $\|.\|_{\mathcal C_T}$ is the uniform norm on $\mathcal C_T$:
$$\|f-g\|_{\mathcal C_T}=\sup_{t\in[0,T]}|f(t)-g(t)|.$$
More generally, for $t\in[0,T]$, we set $$D_t(\mu_1,\mu_2)=\inf_{\pi\in\Pi}\int_{\C\times\C} 1\wedge\sup_{s\in[0,t]}|x_s-y_s|\dx
  \pi(x,y).$$

Endowed with the Wasserstein metric, the space
$\mathcal P(\mathcal C_T)$ of probability measures on $\C$ is complete.
In order to apply a fixed point argument, we will need the following
contraction lemma.

\begin{lem}\label{lem:contraction}Consider both the compact and non
  compact case. Let $T$ be a positive time. Under Assumption
  \ref{hyp:V}, there is a positive constant $K$, not depending on $t$, satisfying
$$D_t(\Phi_T(\mu_1),\Phi_T(\mu_2))\leq K\int_0^tD_s(\mu_1,\mu_2)\dx s,$$ for all $t$
in $[0,T]$ and for all probability measures $\mu_1$ and $\mu_2$ in
$\mathcal P(\mathcal C_T)$.
\end{lem}
\begin{proof}Let $\mu_1$ and $\mu_2$ be two probability measures on $\C$.
For $i=1,2$, define $X_{t,i}$ by
$$\dx X_{t,i}=\left(-\nabla
V(X_{t,i})+\frac{\displaystyle\int_\C\phieps(X^1_{t,i}-x^1_t)\partial_1V(x_t)\dx
 \mu_i(x)}{\displaystyle\int_\C \phieps(X^1_{t,i}-x^1_t)\dx
 \mu_i(x)}e_1\right)\dx t+\sqrt2\dx W_t$$
with given initial condition $X_{0,i}=X_0$, for $i=1,2$.

Notice that 
\begin{equation}\label{eq:fraction}
\displaystyle \frac{\int_\C\phieps(\cdot-x^1_t)\partial_1V(x_t)\dx
 \mu_i(x)}{\int_\C \phieps(\cdot-x^1_t)\dx
 \mu_i(x)}=\frac{\phieps*\mu_{i,t}^{\partial_1V}}{\phieps*\mu_{i,t}^1},
\end{equation}
 and that from \eqref{eq:tete_phi} and Assumption \ref{hyp:V}, the numerator
and the denominator of \eqref{eq:fraction} are respectively bounded from
above and from below by positive constants depending only on $\eta$ and $V$.
Then, for any $x,y$ and $0\leq s\leq T$,
$$\left|\frac{\phieps*\mu_{1,s}^{\partial_1V}}{\phieps*\mu_{1,s}^1}(x)
-\frac{\phieps*\mu_{2,s}^{\partial_1V}}{\phieps*\mu_{2,s}^1}(y)\right|\leq K\left(|x-y|\wedge1+D_s(\mu_1,\mu_2)\right).$$
Consequently,
$$\E\left[1\wedge \|X_1-X_2\|_{\mathcal C_t}\right]\leq K\left(\int_0^t\E\left[1\wedge \|X_1-X_2\|_{\mathcal C_s}\right]\dx
s+\int_0^tD_s(\mu_1,\mu_2)\dx s\right),$$for all $t\leq T$.
Using Gr\"onwall's Lemma, we then find, for any $t\leq T$,

$$\E\left[1\wedge\|X_1-X_2\|_{\mathcal C_t}\right]\leq K\int_0^tD_s(\mu_1,\mu_2)\dx s.$$
 But $$D_t(\Phi_t(\mu_1),\Phi_t(\mu_2))\leq\E\left[1\wedge\|X_1-X_2\|_{\mathcal C_t}\right]$$ since $X_1$
and $X_2$ respectively have $\Phi_t(\mu_1)$ and $\Phi_t(\mu_2)$ as
distributions, finishing the proof.
\end{proof}

\begin{proof}[Proof of Theorem~\ref{th:ex_un_pnl}]
Iterating Lemma~\ref{lem:contraction}, we find existence and uniqueness of a fixed point of
$\Phi_T$, given~$X_0$, which yields uniqueness of the distribution $P$ of the solution to
\eqref{eq:EDS_eps} on $[0,T]$. 

The law $P$ of any solution being unique, we can substitute the marginal
of $P$ at time $t$ in Equation~\eqref{eq:EDS_eps}, and we obtain a
linear stochastic differential equation
with Lipschitz continuous coefficients. Pathwise uniqueness holds for
that kind of equation, so that weak existence and pathwise  uniqueness hold for
\eqref{eq:EDS_eps}. Consequently, from Yamada-Watanabe Theorem, it admits a unique strong solution.

\end{proof}

\subsection{Convergence to the nonlinear process}
We are now going to let $\varepsilon$ and $\alpha$ go to $0$ in
\eqref{eq:EDS_eps}. 

We denote by $X_t^\eta$ the unique strong solution to \eqref{eq:EDS_eps}, with initial
condition $X_0$ and Brownian motion $W^n$ replaced with $W.$ The
distribution of $\{X_t^\eta\}$ will be denoted $P^\eta$. We expect a possible
limit $X$ of $X^\eta$ as $\eta$ goes to $0$ to be
a solution to \eqref{eq:EDS}. To this aim, we define 
the following martingale problem:

\begin{defi}\label{def:prob_mart}
We say that a probability measure $P$ on the space $\C$ of continuous
paths is a solution to the martingale problem associated to
\eqref{eq:EDS} if its time marginals $P_t$ admit a density $p_t$ with respect
to the Lebesgue measure, and if,
under the measure $P$,
\begin{itemize}
\item the canonical process $x\in\C$ is such that for
any twice differentiable function which is bounded as well as its first
and second derivatives, the process
\begin{equation}\label{eq:prob_mart}
m_t=\psi(x_t)-\psi(x_0)+\int_0^t\nabla\psi(x_s)\nabla V(x_s)\dx
s-\int_0^t\Delta\psi(x_s)\dx
s-\int_0^t\partial_1\psi(x_s)\frac{p_t^{\partial_1V}(x_t)}{p_t^1(x_t)}\dx s,\end{equation}
is a martingale with respect to the filtration $\sigma(x_s,s\leq t)$.
\item $\{x_0\}$ has law $P_0$.
\end{itemize}\end{defi}
Notice that, since the drift coefficient is bounded, the Girsanov
theorem shows that it is not restrictive to assume that $P_t$ has a density.

We deduce from Theorem \ref{th:unicite_EDS} the following result:
\begin{prop}\label{prop:unicite_prob_mart}
In the compact case under Assumptions \ref{hyp:V} and \ref{hyp:cond_init},
or in the non compact case under Asumptions \ref{hyp:V},
\ref{hyp:cond_init} and \ref{hyp:cond_init_nc}, uniqueness holds for the
martingale problem defined in Definition~\ref{def:prob_mart}.
\end{prop}

Our aim in this section will be to prove the following results:
\begin{theo}\label{th:convergence_eta}
Let Assumptions \ref{hyp:V} and \ref{hyp:cond_init} hold.

In the compact case, $(P^\eta)_{\eta>0}$ converges as $\eta$ goes to
$0$ to the solution of the martingale problem.

In the non compact case, the family of probability
measures $(P^\eta)_{\eta>0}$ is tight, and any converging subsequence
converges to a solution of the martingale problem defined in Definition
\ref{def:prob_mart}. Under the additional Assumption
\ref{hyp:cond_init_nc}, $(P_\eta)_{\eta>0}$ actually converges to the
unique solution.
\end{theo}
As a corollary of Theorem \ref{th:convergence_eta}, one has existence of
solutions to \eqref{eq:EDS} (under regularity assumptions on the
initial condition).

From Proposition \ref{prop:unicite_prob_mart}, in order to prove Theorem \ref{th:convergence_eta}, it is enough to
prove that the family $(P^\eta)_{\eta>0}$ is tight, and that any
converging subsequence converges to a solution of the martingale problem.

Our first step will be to derive the Fokker-Planck equation satisfied by
the distribution of $\{X_t^\eta\}.$ Let~$\psi$ be a smooth bounded function on $\mathcal D$, with bounded derivatives. Applying It\=o's formula to
$\psi(X_t^\eta)$ and taking the expectation, we find that

\begin{equation}
\int_\mathcal D\psi \dx P_T^\eta=\int_\mathcal D\psi
p_0(x)\dx x+\int_0^T\int_\mathcal D\left(\Delta\psi-\nabla\psi\cdot\nabla
 V\right)\dx P_t^\eta\dx
t\label{eq:Ito_FP}+\int_0^T\int_\mathcal D\partial_1\psi\frac{\phieps*P_t^{\eta,\partial_1V}}{\phieps*P_t^{\eta,1}}\dx
P_t^\eta\dx t.\end{equation}
Equation \eqref{eq:Ito_FP} is a weak formulation of the
following partial differential equation 

\begin{equation}\label{eq:EDP_eps}
  \partial_t P_t^\eta=\text{div}\left(P_t^\eta\nabla V+\nabla
    P_t^\eta\right)-\partial_1\left(P_t^\eta\frac{\phieps*P_t^{\eta,\partial_1V}}{\phieps*P_t^{\eta,1}}\right).\end{equation}
We are going to show that $P_t^\eta$, or more precisely, its density, is actually a solution to equation
\eqref{eq:EDP_eps} in the following stronger sense.

\begin{defi}\label{def:sol_EDP_eps}
A function $u$ is said to be a solution to \eqref{eq:EDP_eps} with
initial condition $p_0$ if, in the compact case,
\begin{itemize}
\item $u$ belongs to
  $\LL^\infty((0,T),\LL^2(\T^d))\bigcap\LL^2((0,T),\HH^1(\T^d))$ ;
\item for any function $\psi\in\HH^1(\T^d)$, we have:
\begin{equation}\label{eq:EDP_dist_eps}\partial_t\int_\mathcal D u_t\psi=-\int_\mathcal D u_t\nabla
V\cdot\nabla\psi-\int_\mathcal D\nabla
u_t\cdot\nabla\psi+\int_\mathcal D(\partial_1\psi)u_t\frac{\phieps*u_t^{\partial_1V}}{\phieps*u_t^1}
\end{equation}
in the sense of distributions in time ;
\item $u_0=p_0$.
\end{itemize}
In the non compact case these conditions are replaced by
\begin{itemize}
\item $u$ belongs to
  $\LL^\infty((0,T),\LL^2(w))\bigcap\LL^2((0,T),\HH^1(w))$ ;
\item for any function $\psi\in\HH^1(w)$, we have:
\begin{equation}\label{eq:EDP_dist_eps_nc}\partial_t\int_\mathcal D u_t\psi
  w=-\int_\mathcal D u_t\nabla
V\cdot(w\nabla\psi+\psi\nabla w)-\int_\mathcal D\nabla
u_t\cdot(w\nabla\psi+\psi\nabla w)+\int_\mathcal D(\partial_1\psi)u_t\frac{\phieps*u_t^{\partial_1V}}{\phieps*u_t^1}w
\end{equation}
in the sense of distributions in time ;
\item $u_0=p_0$.
\end{itemize}
\end{defi}
As for Definition \ref{def:sol_EDP}, these conditions make sense.

With this definition, one has the following result:
\begin{lem}\label{lem:u_L2}Consider both the compact and the
  non compact cases. Under  Assumptions~\ref{hyp:V} and~\ref{hyp:cond_init}, the distribution $P^\eta_t$ of $\{X^\eta_t\}$ has a density $p_t^\eta$ with
respect to the Lebesgue measure such that $p^\eta$ satisfies~\eqref{eq:EDP_eps} in the
sense of Definition~\ref{def:sol_EDP_eps}.

Moreover, the family $(p^\eta)_{\eta>0}$ is bounded in
$\LL^\infty((0,T),\LL^2(\mathcal D))\cap\LL^2((0,T),\HH^1(\mathcal D))$ and, in the non compact case, under Assumption
\ref{hyp:cond_init_nc}, $(p^\eta)_{\eta>0}$ is bounded in
$\LL^\infty((0,T),\LL^2(w))\cap\LL^2((0,T),\HH^1(w)).$
\end{lem}

\begin{proof}
Since the drift coefficient in \eqref{eq:EDS_eps} is bounded, following the proof of
Lemmas \ref{lem:p_L2} and \ref{lem:regularite_EDP}, we obtain that $P_t^\eta$ has a density
$p_t^\eta$, where $p^\eta$ satisfies the first condition in Definition
\ref{def:sol_EDP_eps}. Applying It\=o's formula to $\psi(X_t^\eta)$ for
some smooth $\psi$, we find that \eqref{eq:EDP_dist_eps}
(\eqref{eq:EDP_dist_eps_nc} in the non compact case) holds for a smooth $\psi.$ Using the density of smooth
functions in $\HH^1(\Td)$, it holds for any $\psi$ in~$\HH^1(\T^d)$,
and the same is true for $\HH^1(w)$ in the non compact case.

To prove that $p^\eta$ is bounded independently of $\eta,$ notice that from the boundedness of $\nabla V$, the function
$\frac{\phieps*p_t^{\eta,\partial_1V}}{\phieps*p_t^{\eta,1}}$ is bounded
from above uniformly with respect to~$\eta$.
Consequently, from Cauchy-Schwarz inequality,
\begin{align*}\frac12\partial_t\|p_t^\eta\|^2_{\LL^2(\Td)}=&-\|\nabla p_t^\eta\|^2_{\LL^2(\Td)}-\int_\Td
p_t^\eta\nabla p_t^\eta\cdot\nabla
V+\int_\Td(\partial_1p_t^\eta)p_t^\eta
\frac{\phieps*p_t^{\eta,\partial_1V}}{\phieps*p_t^{\eta,1}}\\
\leq&-\|\nabla p_t^\eta\|_{\LL^2(\Td)}^2+K\|p_t^\eta\|_{\LL^2(\Td)}\|\nabla p_t^\eta\|_{\LL^2(\Td)}.
\end{align*}
where, the constant $K$ does not depend on $\eta.$
We finish the proof using Young's inequality, and then Gr\"onwall's
Lemma.

The proof is similar in the non compact case.
\end{proof}

Thanks to Lemma~\ref{lem:u_L2}, we can prove the relative compactness of the family
$p^\eta$ in a nice sense.

\begin{lem}\label{lem:compacite}Consider both the compact and the non
  compact cases. Under Assumptions \ref{hyp:V} and \ref{hyp:cond_init}, for any bounded open domain $\mathcal O$
  in $\mathcal D$, 
the set $({p^\eta}_{|\mathcal O})_{\eta>0}$ of restrictions of the functions 
$p^\eta$ to~$\mathcal O$ is relatively compact in the space
$\LL^2((0,T),\LL^2(\mathcal O))$. Moreover, the set $(P^\eta)_{\eta>0}$
of laws of the solution is tight.\end{lem}

\begin{proof}
We first prove the relative compactness of $p^\eta$ in $\LL^2((0,T),\LL^2(\mathcal
O))$. We use the fact that for a bounded open domain $\mathcal O$ and
for $p,q\in(1,\infty)$, the space
$$E_{p,q}=\{f\in\LL^p((0,T),\HH^1(\mathcal O)) ,\text{ such that }
\partial_t f\in\LL^q((0,T),\HH^{-1}(\mathcal O))\}$$ 
imbeds compactly in $\LL^p((0,T),\LL^2(\mathcal O))$ (see \cite[page 57]{lions-69}).
We already know that the set $(p^\eta)_{\eta>0}$ is bounded in
$\LL^2((0,T),\HH^1(\mathcal D))$, so that the set $({p^\eta}_{|\mathcal O})_{\eta>0}$ is bounded in
$\LL^2((0,T),\HH^1(\mathcal O))$. Thus, it is enough to show that
$(\partial_t{p^\eta}_{|\mathcal O})_{\eta>0}$ is bounded in
$\LL^q((0,T),\HH^{-1}(\mathcal O))$,
for some $q\in(1,\infty)$ to finish the proof.
The following equation holds
$$\partial_t p^\eta=\text{div}(p^\eta\nabla V)+\Delta p^\eta
-\partial_1\left(p^\eta\frac{\phieps*p_t^{\eta,\partial_1V}}{\phieps*p_t^{\eta,1}}\right),$$
showing, since $(p^\eta)_{\eta>0}$ is bounded in
$\LL^2((0,T),\HH^1(\mathcal D))$, that
$(\partial_tp^\eta)_{\eta>0}$ is bounded in
$\LL^2((0,T),\HH^{-1}(\mathcal D))$, thus,
$\partial_t{p^\eta}_{|\mathcal O}$ is bounded in
$\LL^2((0,T),\HH^{-1}(\mathcal O))$. This shows that $(p^\eta_{|\mathcal
O})_{\eta>0}$ is relatively
compact in $\LL^2((0,T),\LL^2(\mathcal O))$.

Now we prove the relative compactness of $(P^\eta)_{\eta>0}$ in
$\mathcal P(\mathcal C_T).$ For this aim, we use Kolmogorov compactness criterion. 
At time $t=0$, $X^\eta_0$ is equal to $X_0$, independently of $\eta$. Consequently, the
family~$(X_0^\eta)_{\eta>0}$ is tight.
To conclude the proof, it is enough to show that for some positive
constants $a$, $b$ and $K$, $$\sup_{\eta>0}\E\left[|X^\eta_t-X^\eta_s|^a\right]\leq K|t-s|^{1+b}$$
for $0\leq s,t\leq T$.
Since $\nabla V$ is bounded, we have, for $0\leq s,t\leq T$ and $p>1$,
\begin{align*}
\E\left[\left|X_t^\eta-X_s^\eta\right|^p\right]^{1/p}&\leq
\E\left[\left|\int_s^t\nabla
V(X_\tau^\eta)\dx\tau\right|^p\right]^{1/p}+\E\left[\left|W_t-W_s\right|^p\right]^{1/p}+\E\left[\left|\int_s^t\frac{\phieps*u_\tau^{\eta,\partial_1V}(X^{\eta,1}_\tau)}{\phieps*u_\tau^{\eta,1}(X^{\eta,1}_\tau)}\dx\tau\right|^p\right]^{1/p}\\
&\leq K\left(\left|t-s\right|+\left|t-s\right|^{1/2}\right).
\end{align*}
This rewrites 
$$\E\left[\left|X_t^\eta-X_s^\eta\right|^p\right]\leq K|t-s|^{p/2},$$
for some positive $K$. Taking $p=3$, Lemma~\ref{lem:compacite} follows.
\end{proof}

As a consequence of Lemma~\ref{lem:compacite}, using a diagonal
argument, we can extract a subsequence of $\eta\rightarrow0$, still
denoted $\eta$ such that:
\begin{itemize}
\item
$p^\eta$ converges almost everywhere on $(0,T)\times\mathcal D$ and in
$\LL^2((0,T),\LL^2(\mathcal O))=\LL^2((0,T)\times\mathcal O)$ as
$\eta$ goes to $0$, for any bounded open domain $\mathcal O$ to a
function $p^0.$ 
\item $P^\eta$ converges in $\mathcal P(\mathcal
C([0,T]))$ as $\eta$ goes to $0$ to a probability measure $P^0.$
\end{itemize}

To let $\eta$ go to zero in \eqref{eq:EDS_eps}, we finally need
the following lemma.
\begin{lem}\label{lem:expr_v_utild}
Consider both the compact and the non compact cases. Under
  Assumptions \ref{hyp:V} and \ref{hyp:cond_init}, the limit
$p^0$ of $p^\eta$ is such that
$p^0_t$ is the density of the time marginal of $P^0$ for almost all
times $t$. Moreover, the convergence of $p^\eta$ to $p^0$ also holds
in $\LL^1((0,T)\times\mathcal D)$ and up to a second subsequence extraction,~$\displaystyle\frac{\phieps*p^{\eta,\partial_1V}}{\phieps*p^{\eta,1}}$
converges almost everywhere on $(0,T)\times\T$ to $\displaystyle\frac{p^{0,\partial_1V}}{p^{0,1}}$ as $\eta$ goes to zero.
\end{lem}
\begin{proof}
We first prove that $p^\eta$ converges to $p^0$ in
$\LL^1((0,T)\times\mathcal D)$. It holds
\begin{align*}
\int_0^T\int_{\mathcal D}|p^\eta-p^0|&=\int_0^T\int_{\mathcal
  D}(p^\eta-p^0)+2\int_0^T\int_{\mathcal D}(p^\eta-p^0)^-\\
&=2\int_0^T\int_{\mathcal D}(p^\eta-p^0)^-.
\end{align*}
But $p^\eta$ converges almost everywhere to $p^0$, and
$(p^\eta-p^0)^-$ is bounded from above by the integrable function
$p^0$. Consequently, by the Lebesgue theorem, $p^\eta$ converges to $p^0$
in $\LL^1((0,T)\times\mathcal D)$.

A consequence of this convergence and of the boudedness of $V$ is that the sequences
$(p^{\eta,\partial_1V})_{\eta>0}$ and $(p^{\eta,1})_{\eta>0}$ converge
in~$\LL^1((0,T)\times\T)$ respectively to $p^{0,\partial_1V}$ and
$p^{0,1}$.

As a consequence, $\phieps*p^{\eta,1}$ and
$\phieps*p^{\eta,\partial_1V}$ also converge in $\LL^1((0,T)\times\T)$
to the same limits. Therefore, up to the extraction of a
second subsequence, we have pointwise convergence almost everywhere for
the denominator and the numerator of $\displaystyle\frac{\phieps*p^{\eta,\partial_1V}}{\phieps*p^{\eta,1}}$.

Now we show that $p_t^0$ is for almost all $t$ the density of the time
marginal $P^0_t$ of $P^0.$
Since $P^\eta$ converges to $P^0$ as
  $\eta$ goes to $0$ in~$\mathcal P(\mathcal C_T)$, then
  $\E\left[\Psi(X^\eta)\right]$ converges to $\E\left[\Psi(X^0)\right]$
  as $\eta$ goes to $0$, for any bounded continuous
  functional $\Psi$ on $\mathcal C_T$. Taking a
  function of the form
  $\Psi(Y)=\int_0^T\theta(t)\tilde\Psi(Y_t)\dx t$ where $\tilde\Psi$ and $\theta$ are
  bounded and continuous, one has
$$\E[\Psi(X^\eta)]=\int_0^T\theta(t)\left(\int_\Td\tilde\Psi p_t^\eta\right)\dx t.$$
Moreover, since $p^\eta$ converges to $p^0$ in $\LL^1((0,T)\times\Td)$,
 one has 
$$\int_0^T\left(\theta(t)\int_\Td\tilde\Psi p_t^\eta\right)\dx t\rightarrow_{\eta\rightarrow0}\int_0^T\theta(t)\left(\int_\Td\tilde\Psi p_t^0\right)\dx t.$$
As a result, 
$$\E\left[\int_0^T\theta(t)\tilde\Psi(X_t^0)\dx t\right]=\int_0^T\theta(t)\left(\int_\Td\tilde\Psi p_t^0\right)\dx t,$$
so that, almost everywhere, $p_t^0$ is the time marginal of $P^0$.
\end{proof}

We can now prove Theorem \ref{th:convergence_eta}. We want to prove that
$P^0$ is a solution to the martingale problem defined in Definition
\ref{def:prob_mart}. It is enough to show that for $0\leq
s_1\leq\hdots\leq s_n\leq s\leq t$, any bounded continuous function~$g$
and any twice differentiable function $\psi$ with bounded derivatives,
one has $\int_{\C}g(x_{s_1},\hdots,x_{s_n})(m_t-m_s)\dx P^0=0.$

Under the probability measure $P^\eta$, the canonical process $x\in\mathcal
C([0,T])$ is such that 

$$m_t^\eta=\psi(x_t)-\psi(x_0)-\int_0^t\Delta\psi(x_s)\dx s+\int_0^t\nabla
V(x_s)\nabla\psi(x_s)\dx s-\int_0^t\partial_1\psi(x_s)\frac{\phieps*p_s^{\eta,\partial_1V}(x_s^1)}{\phieps*p_s^{\eta,1}(x_s^1)}\dx s$$
is a martingale. We thus have
$$\int_\C g(x_{s_1},\hdots,x_{s_n})(m_t^\eta-m_s^\eta)\dx P^\eta=0.$$
Consequently, denoting $\tilde\eta=(\tilde\varepsilon,\tilde\alpha)$

\begin{align*}
\left|\int_\C g(x_{s_1},\hdots,x_{s_n})(m_t-m_s)\dx P^0\right|&\leq\left|\int_\C
  g(x_{s_1},\hdots,x_{s_n})(m_t^{\tilde\eta}-m_s^{\tilde\eta})\dx P^\eta\right|\\
&\quad+\left|\int_\C g(x_{s_1},\hdots,x_{s_n})(m^{\tilde\eta}_t-m_s^{\tilde\eta})\dx(P^\eta-P^0)\right|\\
&\quad+\left|\int_\C
  g(x_{s_1},\hdots,x_{s_n})\big((m^{\tilde\eta}_t-m_s^{\tilde\eta})-(m_t-m_s)\big)\dx P^0\right|.
\end{align*}
Taking $\lim\sup_{{\tilde\eta}\rightarrow 0}\lim\sup_{\eta\rightarrow0}$,
we obtain:

\begin{equation}\label{eq:limsup}
\left|\int_\C
    g(x_{s_1},\hdots,x_{s_n})(m_t-m_s)\dx P^0\right|\leq\lim\sup_{{\tilde\eta}\rightarrow 0}\lim\sup_{\eta\rightarrow0}\left|\int_\C
  g(x_{s_1},\hdots,x_{s_n})(m_t^{\tilde\eta}-m_s^{\tilde\eta})\dx P^\eta\right|.
\end{equation}
Indeed, $g(x_{s_1},\hdots,x_{s_n})(m_t^{\tilde\eta}-m_s^{\tilde\eta})$ is a bounded
continuous function of $x$, and $P^\eta$ converges to
$P^0$. Moreover, we have 

\begin{align*}
&\left|\int_\C
  g(x_{s_1},\hdots,x_{s_n})\big((m^{\tilde\eta}_t-m_s^{\tilde\eta})-(m_t-m_s)\big)\dx P^0\right|\\
=&\left|\int_\C\int_s^tg(x_{s_1},\hdots,x_{s_2})\partial_1\psi(x_\tau)
\left[\frac{\phiepss*p_\tau^{{\tilde\eta},\partial_1V}}{\phiepss*p_\tau^{{\tilde\eta},1}}-\frac{p_\tau^{0,\partial_1V}}
{p_\tau^{0,1}}\right](x_\tau^1)~\dx\tau~\dx P^0(x)\right|\\
\leq&K\int_s^t\int_\mathcal D\left|\left[\frac{\phiepss*p_\tau^{{\tilde\eta},\partial_1V}}
{\phiepss*p_\tau^{{\tilde\eta},1}}-\frac{p_\tau^{0,\partial_1V}}{p_\tau^{0,1}}\right](y)\right|p_\tau^0(y)~\dx y~\dx\tau.
\end{align*}
This last integral goes to $0$ as ${\tilde\eta}$ goes to $0$, since the function 
$\left[\frac{\phieps*p_\tau^{{\tilde\eta},\partial_1V}}
{\phiepss*p_\tau^{{\tilde\eta},1}}-\frac{p_\tau^{0,\partial_1V}}{p_\tau^{0,1}}\right]$
converges almost everywhere to $0$ on $[s,t]\times\mathcal D$,
and is bounded fom above by some positive constant.
To conclude, we estimate the right hand side in \eqref{eq:limsup}:
\begin{align*}
&\left|\int_\C g(x_{s_1},\hdots,x_{s_n})(m_t^{\tilde\eta}-m_s^{\tilde\eta})\dx
P^\eta(x)\right|\\
=&\left|\int_\C g(x_{s_1},\hdots,x_{s_n})((m_t^{\tilde\eta}-m_s^{\tilde\eta})-(m_t^\eta-m_s^\eta))\dx
P^\eta(x)\right|\\
=&\left|\int_{\C}
  g(x_{s_1},\hdots,x_{s_n})\int_s^t\partial_1\psi(x_\tau)\left(\frac{\phieps*p_\tau^{\eta,\partial_1V}}
{\phieps*p_\tau^{\eta,1}}-\frac{\phiepss*p_\tau^{{\tilde\eta},\partial_1V}}{\phiepss*p_\tau^{{\tilde\eta},1}}\right)(x_\tau^1)
\dx\tau\dx P^\eta(x)\right|\\
\leq&K\int_s^t\int_\C\left|\frac{\phieps*p_\tau^{\eta,\partial_1V}(x_\tau^1)}
{\phieps*p_\tau^{\eta,1}(x_\tau^1)}-\frac{\phiepss*p_\tau^{{\tilde\eta},\partial_1V}(x_\tau^1)}
{\phiepss*p_\tau^{{\tilde\eta},1}(x_\tau^1)}\right|\dx P^\eta(x)\dx\tau\\
=&K\int_s^t\int_\T\left|\frac{\phieps*p_\tau^{\eta,\partial_1V}(y)}{\phieps*p_\tau^{\eta,1}(y)}
-\frac{\phiepss*p_\tau^{{\tilde\eta},\partial_1V}(y)}{\phiepss*p_\tau^{{\tilde\eta},1}(y)}\right|p_\tau^{\eta,1}(y)\dx y\dx\tau.
\end{align*}
This last integral tends to $0$ as $\eta$ and ${\tilde\eta}$ go to $0$,
since $p^{\eta,1}$ converges in $\LL^1((s,t)\times\T)$, and since
$\displaystyle\frac{\phieps*p_\tau^{\eta,\partial_1V}(y)}{\phieps*p_\tau^{\eta,1}(y)}
-\frac{\phiepss*p_\tau^{{\tilde\eta},\partial_1V}(y)}{\phiepss*p_\tau^{{\tilde\eta},1}(y)}$ is bounded and converges almost everywhere to $0$.
We then obtain Theorem \ref{th:convergence_eta}.

\subsection{Another existence result for the nonlinear process}
From Theorem \ref{th:convergence_eta}, we know that existence
holds for
\eqref{eq:EDS} under some regularity assumptions on the initial
condition. Indeed, if $P^0$ is the limit of some subsequence of
$P^\eta$, then the canonical process~$x$ defined on the canonical space
$(\C,P^0)$ is a solution to Equation \eqref{eq:EDS}. By approximating the initial condition by regular densities,
one can relax the regularity assumption.
\begin{theo}\label{th:existence_EDS}
Consider both the compact and non compact cases. Under Assumption
\ref{hyp:V}, weak existence holds for Equation
\eqref{eq:EDS} with given initial condition~$X_0$. Moreover, for
positive $s$, the law of
$\{X_s\}$ has a density $p_s$ such that, for $0<t<T$,
$$p\in\LL^\infty((t,T),\LL^2(\mathcal
D))\bigcap\LL^2((t,T),\HH^1(\mathcal D)).$$
\end{theo}
Notice that, under the hypotheses of Theorem \ref{th:existence_EDS}, we have no uniqueness result.

\begin{proof}
Theorem \ref{th:convergence_eta} yields existence for \eqref{eq:EDS} when
the initial condition satisfies Assumption~\ref{hyp:cond_init}. To prove
existence without assumption on the initial condition, we use
approximations of the initial condition. Let $(p_0^k)_{k\in\mathbb N}$ be a sequence of probability
densities satisfying Assumption \ref{hyp:cond_init} and converging to
$p^0$ in $\mathcal P(\mathcal D)$ (for example, this sequence can be obtained by convolution
with a gaussian kernel). From Theorem~\ref{th:convergence_eta}, there
exists a solution $(X_t^k)$ to Equation \eqref{eq:EDS} driven by a
Brownian motion~$W$ defined on some
probability space $(\Omega,\mathcal F,\mathbb P)$, such that $X_0^k$
admits $p_0^k$ as density.

As in the proof of Lemma \ref{lem:compacite}, we can apply Kolmogorov
criterion, so that the family of distributions $P^k$ of
$(\{X^k_t\})_{0\leq t\leq T}$ is tight.
Consequently, we can extract from $(P^k)$ a converging subsequence whose
limit is denoted $P$. To prove that $P$ satisfies the martingale problem
defined in Definition~\ref{def:prob_mart}, we need some estimate on the time marginals of $P^k$,
uniformly in $k$.

According to Lemma \ref{lem:regularite_EDP}, the law of $\{X_t^k\}$ has a
density $p_t^k$ such that
$p^k$ lies in $\LL^\infty((0,T),\LL^2(\mathcal D))$ and 
$\LL^2((0,T),\HH^1(\mathcal D))$.
Notice that the drift coefficient 
$b^k_t(X_t)\stackrel{\rm def}=-\nabla V(X_t)+\E[\partial_1V(X_t)|\{X_t^1\}]e_1$ in Equation \eqref{eq:EDS} is bounded, so that we can apply the Girsanov
Theorem. Indeed, define $$L_t^k=\exp\left(-\frac1{\sqrt2}\int_0^tb^k_s(X_s^k)\dx
  W_s-\frac14\int_0^t\|b^k_s(X_s^k)\|^2\dx s\right).$$
Novikov's Condition is satisfied for this process, so that the formula
$$\mathbb Q_k(A)=\E[\mathbf 1_AL_t^k],$$ for $A\in\sigma(W_s)_{s\leq t}$, defines a
probability distribution $\mathbb Q_k$
on $\Omega$ such that, under $\mathbb Q_k$, the process
$$\frac1{\sqrt2}\left(X_t^k-X_0^k\right)=W_t+\frac1{\sqrt2}\int_0^tb_s^k(X_s^k)\dx
s$$ is a Brownian motion.
Denote $\gamma_t^k$ the law of $\{X_t^k\}$ under $\mathbb Q_k$. Notice
that since, under $\mathbb Q_k$, $X_t^k$ is the sum of $\sqrt2$ times a
Brownian motion at time $t$ and an independent random variable
$X_0^k$, $\gamma_t^k$ has a
density with respect to the Lebesgue measure which is bounded by $\frac
K{t^{d/2}}$ where $K$ is a constant not depending on $k$ and $t.$
As a result, for a given
function $\psi$ in $\LL^2(\mathcal D)$, one has
\begin{align*}
\left|\int_\mathcal D \psi(x)\dx P_t^k(x)\right|&=\left|\E\left[\psi\left(X_t^k\right)\right]\right|\\
&=\left|\E_{\mathbb Q_k}\left[\psi\left(X_t^k\right)\left(L_t^k\right)^{-1}\right]\right|\\
&\leq\left(\int_\mathcal D\psi^2\dx \gamma_t^k\right)^{\frac12}\E\left[(L_t^k)^{-2}\right]^{1/2}\\
&\leq \frac K{t^{d/4}}\left\|\psi\right\|_{\LL^2(\mathcal D)},
\end{align*}
where $K$
is a positive constant, which does not depend on $k$ since $\left|b^k\right|$ is bounded
from above by~$\|\nabla V\|_{\LL^\infty}$.
Consequently, for any $0<t<T$, $\|p^k_s\|_{\LL^2(\mathcal D)}$ is
bounded uniformly in $k$ and in $s\in[t,T]$.
Moreover, since $p^k$ is a solution to Equation \eqref{eq:EDP} in the sense
of Definition~\ref{def:sol_EDP} it holds, from \eqref{eq:EDP_dist}
$$\partial_t\|p_s^k\|_{\LL^2(\mathcal D)}^2\leq-\|\nabla
p_s^k\|_{\LL^2(\mathcal D)}^2+K\|p_s^k\|_{\LL^2(\mathcal D)}^2,$$
so that $(p^k)_{k\in\mathbb N}$ is also bounded in $\LL^2((t,T),\HH^1(\mathcal D))$.
Adapting the proof of
Lemma~\ref{lem:compacite}, we find that the family $(p^k_{|\mathcal O})$ is compact in
$\LL^2((t,T),\LL^2(\mathcal O))$ for any open subset $\mathcal O$
of $\mathcal D$.
By a diagonal argument, and using the proof of Lemma \ref{lem:expr_v_utild} we can thus construct a subsequence $k_n$ such
that 
\begin{itemize}
\item $P^{k_n}$ converges to a probability measure $P^0$ whose time
  marginals $P_t$ have a density $p_t^0$, for all~$t>0$,
\item $p^{k_n}$ converges almost everywhere on $(0,T)\times\mathcal D$ and in $\LL^1((0,T)\times\mathcal D)$ to $p^0$,
\item $\displaystyle\frac{p^{k_n,\partial_1V}}{p^{k_n,1}}$ converges almost everywhere on
  $(0,T)\times\mathcal D$ to $\displaystyle\frac{p^{0,\partial_1V}}{p^{0,1}}$.
\end{itemize}
Then, adapting the proof of
Theorem \ref{th:convergence_eta}, we see that $P^0$ solves the martingale problem.
\end{proof}

\subsection{Rate of convergence}

We are going to exhibit a control on the rate of the convergence of $p^\eta$ to
$p$. Moreover, we give an estimate of the difference between $\displaystyle\frac{\phieps*p^{\eta,\partial_1V}}{\phieps*p^{\eta,1}}$
and the biasing force $A'_t=\displaystyle\frac{p^{\partial_1V}_t}{p^1_t}$ which is the quantity one is interested in practice.
\begin{theo}\label{th:vitesse}
Under Assumptions \ref{hyp:V} and \ref{hyp:cond_init}, it
holds, in the compact case,
$$\|p^\eta-p\|_{\LL^\infty((0,T),\LL^2(\T^d))}+\|p^\eta-p\|_{\LL^2((0,T),\HH^1(\T^d))}\leq
K(\alpha+\sqrt\varepsilon),$$
and, in the non compact case, under the additional Assumption \ref{hyp:cond_init_nc}
$$\|p^\eta-p\|_{\LL^\infty((0,T),\LL^2(w))}+\|p^\eta-p\|_{\LL^2((0,T),\HH^1(w))}\leq
K(\alpha+\sqrt\varepsilon),$$
 for some positive constant $K$ not
depending on $\alpha$ and $\varepsilon$.
Moreover, we have the following bound on the estimation of the biasing force:
$$\left\|\frac{\phieps*p_t^{\eta,\partial_1V}}{\phieps*p_t^{\eta,1}}
-\frac{p_t^{\partial_1V}}{p_t^1}\right\|_{\LL^2((0,T),\LL^\infty(\T))}\leq K\left(\alpha+\sqrt\varepsilon\right).$$
\end{theo}

\begin{proof}
We give the proof in the non compact case, the one in the compact case
being very similar.
Similar calculations as in the proof of Theorem~\ref{th:unicite_EDP} yield:
\begin{align*}
\frac12\partial_t\|p_t-p^\eta_t\|_{\LL^2(w)}^2+\|\nabla
p_t-\nabla
p_t^\eta\|_{\LL^2(w)}^2&\leq K\|p_t-p^\eta_t\|_{\LL^2(w)}\|\nabla
p_t-\nabla p^\eta_t\|_{\LL^2(w)}+K\|p_t-p_t^\eta\|^2_{\LL^2(w)}\\
&\quad+\|\nabla p_t^\eta-\nabla
p_t\|_{\LL^2(w)}\left\|p_t\frac{p_t^{\partial_1V}}{p_t^1}
-p_t^\eta\frac{p_t^{\eta,\partial_1V}*\phieps}{p_t^{\eta,1}*\phieps}\right\|_{\LL^2(w)}.
\end{align*}
We now estimate the last term.
\begin{align*}
\left\|p_t\frac{p_t^{\partial_1V}}{p_t^1}-p_t^\eta\frac{p_t^{\eta,\partial_1V}*\phieps}
{p_t^{\eta,1}*\phieps}\right\|_{\LL^2(w)}&\leq\left\|\frac{p_t}{p_t^1}
\left(p^{\partial_1V}_t-\phieps*p_t^{\eta,\partial_1V}\right)\right\|_{\LL^2(w)}\\
&\quad+\left\|p_t~\phieps*p_t^{\eta,\partial_1V}\left(\frac1{p_t^1}-\frac1{\phieps*
      p_t^{\eta,1}}\right)\right\|_{\LL^2(w)}+\left\|(p_t-p_t^\eta)\frac{\phieps*p_t^{\eta,\partial_1V}}{\phieps*p_t^{\eta,1}}\right\|_{\LL^2(w)}\\
&\leq\left\|p_t\right\|_{\LL^2(w)}\left\|\frac1{p_t^1}\left(p^{\partial_1V}_t-\phieps*p_t^{\eta,\partial_1V}\right)\right\|_{\LL^\infty(\T)}\\
&\quad+\left\|p_t\right\|_{\LL^2(w)}\left\|\frac{\phieps*p_t^{\eta,\partial_1V}}{p_t^1(\phieps*p_t^{\eta,1})}\left(\phieps*
      p_t^{\eta,1}-p_t^1\right)\right\|_{\LL^\infty(\T)}+\left\|(p_t-p_t^\eta)\frac{\phieps*p_t^{\eta,\partial_1V}}{\phieps*p_t^{\eta,1}}\right\|_{\LL^2(w)}.\\
\end{align*}
From Lemma~\ref{lem:chaleur}, $p_t^1$ is bounded from below uniformly in
time. Using this together with the facts that~$\partial_1V$ is bounded and $p\in\LL^\infty((0,T),\LL^2(w))$,
one obtains
\begin{align*}
\left\|p_t\frac{p_t^{\partial_1V}}{p_t^1}-p_t^\eta\frac{p_t^{\eta,\partial_1V}*\phieps}
{p_t^{\eta,1}*\phieps}\right\|_{\LL^2(w)}&\leq
K\left(\|p_t^{\partial_1V}-\phieps*p_t^{\eta,\partial_1V}\|_{\LL^\infty(\T)}+
\|p_t^1-\phieps*p_t^{\eta,1}\|_{\LL^\infty(\T)}+\|p_t-p_t^\eta\|_{\LL^2(w)}\right).
\end{align*}
Consequently, we have to estimate
$\|p_t^1-\phieps*p_t^{\eta,1}\|_{\LL^\infty(\T)}$ and $\|p_t^{\partial_1V}-\phieps*p_t^{\eta,\partial_1V}\|_{\LL^\infty(\T)}.$
It holds, for $\gamma\in(0,1/2),$
\begin{align*}
\|p_t^{\partial_1V}-\phieps*p_t^{\eta,\partial_1V}\|_{\LL^\infty(\T)}
&\leq\|\phieps*(p_t^{\partial_1V}-p_t^{\eta,\partial_1V})\|_{\LL^\infty(\T)}+\|p_t^{\partial_1V}-\phieps*p_t^{\partial_1V}\|_{\LL^\infty(\T)}\\
&\leq K\alpha+\|p_t^{\partial_1V}-p_t^{\eta,\partial_1V}\|_{\LL^\infty(\T)}+\|p_t^{\partial_1V}-\phieps*p_t^{\partial_1V}\|_{\LL^\infty(\T)}\\
&\leq K\left(\alpha+\|p_t^{\partial_1V}-p_t^{\eta,\partial_1V}\|_{\HH^1(\T)}^{1/2+\gamma}\|p_t^{\partial_1V}-p_t^{\eta,\partial_1V}\|_{\LL^2(\T)}^{1/2-\gamma}\right)+\|p_t^{\partial_1V}-\phieps*p_t^{\partial_1V}\|_{\LL^\infty(\T)}\\
&\leq K\left(\alpha+\|p_t-p_t^\eta\|_{\HH^1(w)}^{1/2+\gamma}\|p_t-p_t^\eta\|_{\LL^2(\T)}^{1/2-\gamma}\right)+\|p_t^{\partial_1V}-\phieps*p_t^{\partial_1V}\|_{\LL^\infty(\T)}.
\end{align*}
Likewise, we have
$$\|p_t^1-\phieps*p_t^{\eta,1}\|_{\LL^\infty(\T)}\leq K\alpha+K\|p_t-p_t^\eta\|_{\HH^1(w)}^{1/2+\gamma}\|p_t-p_t^\eta\|_{\LL^2(\T)}^{1/2-\gamma}+\|p_t^1-\phieps*p_t^1\|_{\LL^\infty(\T)}.$$
To conclude, notice that, in view of Lemma \ref{lem:d1V_continu},
$p_t^{\partial_1V}$ lies in $\HH^1(\T)$. Thus $p_t^{\partial_1V}$ is
H\"older continuous with exponent $1/2$ and
constant~$C\|p_t^{\partial_1V}\|_{\HH^1(\T)}$ (see \cite[Corollaire
IX.13]{brezis-83}). Consequently, since $\psi_\varepsilon\equiv0$ outside~$[-\varepsilon,\varepsilon],$
\begin{align}\label{eq:maj_alpha_eps}
\left|p_t^{\partial_1V}(x)-\phieps*p_t^{\partial_1V}(x)\right|
&=\left|\alpha\int_\T p_t^{\partial_1V}(x)\dx x+\int_\T\psi_\varepsilon(y)\big(p_t^{\partial_1V}(x)-p_t^{\partial_1V}(x-y)\big)\dx
y\right|\nonumber\\
&\leq K\left(\alpha+\|p_t^{\partial_1V}\|_{\HH^1(\T)}\int_\T\psi_\varepsilon(y)\sqrt y\dx
y\right)\nonumber\\
&\leq K\left(\alpha+\|p_t^{\partial_1V}\|_{\HH^1(\T)}\sqrt\varepsilon\int_\T\psi_\varepsilon(y)\dx
y\right)\nonumber\\
&\leq K\left(\alpha+\sqrt\varepsilon\|p_t\|_{\HH^1(w)}\right).
\end{align}
The same inequality holds
for $p^1$
$$\left|p_t^1(x)-\phieps*p_t^1(x)\right|\leq K\left(\alpha+\sqrt\varepsilon\|p_t\|_{\HH^1(w)}\right).$$
Gathering all the previous inequalities, we obtain,
\begin{align*}
\frac12\partial_t\|p_t-p_t^\eta\|_{\LL^2(w)}^2+\|\nabla p_t-\nabla
p_t^\eta\|_{\LL^2(w)}^2&\leq K\|p_t-p_t^\eta\|_{\LL^2(w)}\|\nabla
p_t-\nabla p_t^\eta\|_{\LL^2(w)}+K\|p_t-p_t^\eta\|^2_{\LL^2(w)}\\
&\quad+K\|p_t-p_t^\eta\|_{\HH^1(w)}^{1/2+\gamma}\|p_t-p_t^\eta\|_{\LL^2(\T)}^{1/2-\gamma}\|\nabla
p_t-\nabla
p_t^\eta\|_{\LL^2(w)}\\
&\quad+K\left(\alpha+\sqrt\varepsilon\|p_t\|_{\HH^1(w)}\right)\|\nabla
p_t-\nabla
p_t^\eta\|_{\LL^2(w)}\\
&\quad+K\alpha\|\nabla
p_t-\nabla
p_t^\eta\|_{\LL^2(w)}.
\end{align*}
Consequently, from Young's inequality,
$$\partial_t\|p_t-p_t^\eta\|_{\LL^2(w)}^2+\|\nabla
p_t^\eta-\nabla p_t\|_{\LL^2(w)}^2\leq K\left(\|
p_t-p_t^\eta\|_{\LL^2(w)}^2+\alpha^2+\varepsilon\|p_t\|_{\HH^1(w)}^2\right).$$
Gr\"onwall's Lemma yields the first statement of Theorem \ref{th:vitesse}, noticing that $p$ lies in $\LL^2((0,T),\HH^1(w))$.

For the second statement, arguing as we did above, it holds that
\begin{align*}
\left\|\frac{\phieps*p_t^{\eta,\partial_1V}}{\phieps*p_t^{\eta,1}}-\frac{p_t^{\partial_1V}}{p_t^1}\right\|_{\LL^\infty(\T)}\leq&\quad
K\left(\left\|p^1_t-\phieps*p_t^{\eta,1}\right\|_{\LL^\infty(\T)}
+\left\|\phieps*p_t^{\eta,\partial_1V}-p_t^{\partial_1V}\right\|_{\LL^\infty(\T)}\right)\\
\leq&\quad K\left(\alpha+\sqrt\varepsilon\|p_t\|_{\HH^1(\T)}+\|p_t-p_t^\eta\|_{\HH^1(w)}\right).
\end{align*}
We finish the proof by squaring this inequality and then integrating.
\end{proof}

\section{An interacting particle system approximation}\label{IPS}
In this section, we prove the convergence of the interacting particle
system to the regularized nonlinear processes, and we estimate the
difference between the regularized biasing force
$\frac{\phieps*p_t^{\eta,\partial_1V}}{\phieps*p_t^{\eta,1}}$ and its
particle approximation.

\begin{theo}\label{th:conv_N}Let $T$ be a positive time. Under
  Assumptions \ref{hyp:V}, \ref{hyp:cond_init}, \ref{hyp:cond_init_part} and \ref{hyp:phi},
the solution $(X_{t,n,N}^\eta)_{N\geq1}$ of~\eqref{eq:EDS_N_eps} with initial
condition $X_{0,n,N}^\eta=X_{0,n}$ converges to the solution $\bar X_{t,n}^\eta$ to
\eqref{eq:EDS_eps} with initial condition $X_{0,n}$ in the following
sense: for all $1\leq n\leq N$, and for $\varepsilon, \alpha\leq1$, 
$$\E\left[\sup_{t\in[0,T]}\left|\bar
X_{t,n}^\eta-X_{t,n,N}^\eta\right|\right]<\frac1{\sqrt N}e^{\frac K{\alpha\varepsilon^2}},$$
$K$ being some constant
not depending on $\alpha$, $\varepsilon$ and $N$.

Moreover, one has
\begin{equation}\label{eq:theo}
\E\left[\sup_{t\in[0,T],x^1\in\T}\left|\frac{\phieps*p^{\eta,\partial_1V}}{\phieps*p^{\eta,1}}(x^1)
-\frac{\sum_{n=1}^N\phieps(x^1-X_{t,n,N}^{\eta,1})\partial_1V(X_{t,n,N}^\eta)}{\sum_{n=1}^N\phieps(x^1-X_{t,n,N}^{\eta,1})}\right|\right]
\leq\frac1{\sqrt N}e^{\frac
  K{\alpha\varepsilon^2}}.\end{equation}
\end{theo}

Notice that the the right hand side of \eqref{eq:theo} explodes
when $\varepsilon$ goes to $0$ for a fixed value of $N,$ so that the size of
$\varepsilon$ has to be chosen carefully depending on the value of $N.$
We will also investigate this point numerically in the next section.

To simplify notation, we omit the subscript $N$ and the superscript $\eta$. We first establish the following inequality:

\begin{lem}\label{lem:majoration}We have, for $\varepsilon,\alpha<1$,
  and for any $t$ in $(0,T],$
$$\left|X_{t,n}-\bar X_{t,n}\right|\leq\frac
K{\alpha\varepsilon^2}\int_0^t\left(\left|X_{s,n}-\bar X_{s,n}\right|
+\frac1N\sum_{m=1}^N\left|X_{s,m}-\bar
  X_{s,m}\right|\right)\dx s+K\int_0^tA_s^{n,N}\dx s,$$
where $K$ does not depend on $\alpha$, $\varepsilon$ and $t$, and $A_t^{n,N}$ is defined by
\begin{align*}
A_t^{n,N}&=\frac 1\alpha\left(\left|\frac1N\sum_{m=1}^N\phieps(\bar
X^1_{s,n}-\bar X_{s,m}^1)\partial_1V(\bar
X_{s,m})-\phieps*p^{\eta,\partial_1V}_s(\bar
X^1_{s,n})\right|\right.\\
&\quad+\left.\left|\frac1N\sum_{m=1}^N\phieps(\bar
X^1_{s,n}-\bar X_{s,m}^1)-\phieps*p_s^{\eta,1}(\bar
X^1_{s,n})\right|\right).
\end{align*}
\end{lem}
\begin{proof}
From the definition of $X_{t,n}$ and $\bar X_{t,n}$, we have
\begin{align*}
\left|X_{t,n}-\bar X_{t,n}\right|&\leq\left|\int_0^t\left(\nabla V(X_{s,n})-\nabla V(\bar X_{s,n})\right)\dx
s\right|\\
&\quad+\left|\int_0^t\frac{\sum_{m=1}^N\phieps(X^1_{s,n}-X^1_{s,m})\partial_1V(X_{s,m})}{\sum_{m=1}^N\phieps(X^1_{s,n}-X^1_{s,m})}\dx
s-\int_0^t\frac{\phieps*p_s^{\eta,\partial_1V}(\bar X^1_{s,n})}{\phieps*p_s^{\eta,1}(\bar X^1_{s,n})}\dx s\right|.
\end{align*}
First, $\left|\int_0^t\left(\nabla V(X_{s,n})-\nabla
  V(\bar X_{s,n})\right)\dx
s\right|$ is bounded from above by $K\int_0^t|X_{s,n}-\bar X_{s,n}|\dx s$, since $\nabla V$
is Lipschitz continuous.
Now, we decompose
\begin{align}\label{eq:majoration}
&\left|\frac{\sum_{m=1}^N\phieps(X^1_{s,n}-X^1_{s,m})\partial_1V(X_{s,m})}
{\sum_{m=1}^N\phieps(X^1_{s,n}-X^1_{s,m})}-\frac{\phieps*p_s^{\eta,\partial_1V}(\bar
      X^1_{s,n})}{\phieps*p_s^{\eta,1}(\bar X^1_{s,n})}\right|\nonumber\\
&\leq\left|\frac{\sum_{m=1}^N\phieps(X^1_{s,n}-X^1_{s,m})\partial_1V(X_{s,m})}
{\sum_{m=1}^N\phieps(X^1_{s,n}-X^1_{s,m})}-\frac{\sum_{m=1}^N\phieps(\bar
    X^1_{s,n}-\bar X^1_{s,m})\partial_1V(\bar
    X_{s,m})}{\sum_{m=1}^N\phieps(\bar X^1_{s,n}-\bar X^1_{s,m})}\right|\nonumber\\
&\quad+\left|\frac{\sum_{m=1}^N\phieps(\bar
    X^1_{s,n}-\bar X^1_{s,m})\partial_1V(\bar
    X_{s,m})}{\sum_{m=1}^N\phieps(\bar
    X^1_{s,n}-\bar X^1_{s,m})}-\frac{\phieps*p_s^{\eta,\partial_1V}(\bar
    X^1_{s,n})}{\phieps*p_s^{\eta,1}(\bar X^1_{s,n})}\right|.
\end{align}

Using Assumptions \ref{hyp:V} and \ref{hyp:phi}, the first term in the
right hand-side of \eqref{eq:majoration} can be bounded by
$\frac K{\alpha\varepsilon^2}\left(\left|X_{s,n}-\bar
    X_{s,n}\right|+\frac1N\sum_{m=1}^N|X_{s,m}-\bar X_{s,m}|\right),$ and
the second term in the right hand side of
\eqref{eq:majoration} can be bounded by $KA_t^{n,N}$.
\end{proof}
\begin{proof}[Proof of Theorem \ref{th:conv_N}]
As a consequence of Lemma~\ref{lem:majoration}, we get, for
$\alpha,\varepsilon\leq1$ ,
$$\sup_{t\in[0,T]}|X_{t,n}-\bar
  X_{t,n}|\leq\frac
  K{\alpha\varepsilon^2}\int_0^T\left(\sup_{s\in[0,t]}\left|X_{s,n}-\bar
      X_{s,n}\right|+\frac1N\sum_{m=1}^N\sup_{s\in[0,t]}|X_{s,m}-\bar
    X_{s,m}|\right)\dx t+K\int_0^TA_t^{n,N}\dx t.$$
Taking the expectation, and using the
exchangeability of the couples $(X_n,\bar X_n)_{1\leq n\leq N}$, we get
$$\E\left[\sup_{t\in[0,T]}|X_{t,n}-\bar X_{t,n}|\right]\leq \frac
K{\alpha\varepsilon^2}
\int_0^T\E\left[\sup_{s\in[0,t]}|\bar X_{s,n}-X_{s,n}|\right]\dx t+K\int_0^T\E\left[A_t^{n,N}\right]\dx t.$$
By Gr\"onwall's lemma, one obtains
$$\E\left[\sup_{t\in[0,T]}|X_{t,n}-\bar X_{t,n}|\right]\leq Ke^{\frac K{\alpha\varepsilon^2}T}\int_0^T\E\left[A_t^{n,N}\right]\dx t.$$
To conclude, we estimate $\int_0^T\E[A_t^{n,N}]\dx t$. Let $$\Phi_t^m=\phieps(\bar
      X^1_{t,1}-\bar X^1_{t,m})\partial_1V(\bar
      X_{t,m})-\phieps*p^{\eta,\partial_1V}_t(\bar X^1_{t,1})$$ and $$\Psi_t^m=\phieps(\bar
     X^1_{t,1}-\bar X^1_{t,m})-\phieps*p^{\eta,1}_t(\bar X^1_{t,1}).$$
We have, for $t\leq T$, using again the exchangeability of the couples $(X_n,\bar
X_n)_{1\leq n\leq N}$,
\begin{align*}
\left[\E A_t^{n,N}\right]^2&\leq\E\left[(A_t^{n,N})^2\right]\\
&\leq\frac
K{\alpha^2}\left(\E\left[\left(\frac1N\sum_{m=1}^N\Phi_t^m\right)^2\right]
+\E\left[\left(\frac1N\sum_{m=1}^N\Psi_t^m\right)^2\right]\right)\\
&=\frac
K{N^2\alpha^2}\sum_{m,m'}\left(\E\left[\Phi_t^m\Phi_t^{m'}\right]
+\E\left[\Psi_t^m\Psi_t^{m'}\right]\right).
\end{align*}
But the $\Phi_t^m$ and $\Psi_t^m$ are centered for $m\geq2$, and, for $m\neq m'$,
$\Phi_t^m$ and $\Phi_t^{m'}$, (as well as $\Psi_t^m$ and~$\Psi_t^{m'}$) are
independent conditionally on $\bar X_{t,1}$. Therefore the double products
vanish, and, by exchangeability
$$\left[\E A_t^{n,N}\right]^2\leq\frac{
K(N-1)}{\alpha^2N^2}\left(\E\left[(\Phi_t^2)^2\right]+\E\left[(\Psi_t^2)^2\right]\right)+\frac
K{\alpha^2N^2}(\E\left[(\Phi_t^1)^2\right]+\E\left[(\Psi_t^1)^2\right]).$$
But one has $\E\left[(\Phi_t^2)^2\right]+\E\left[(\Psi_t^2)^2\right]\leq
K\varepsilon^{-2}$ and
$\E\left[(\Phi_t^1)^2\right]+\E\left[(\Psi_t^1)^2\right]\leq K\varepsilon^{-2},$
and the first assertion in Theorem \ref{th:conv_N}
follows.

For the estimation of the force, adapting  the proof of Lemma
\ref{lem:majoration}, we see that
\begin{align*}
&\E\left[\sup_{t\in[0,T]}\left|\frac{\phieps*p_t^{\eta,\partial_1V}}
{\phieps*p_t^{\eta,1}}(x^1)-\frac{\sum_{n=1}^N\phieps(x^1-X_{t,n,N}^1)
\partial_1V(X_{t,n,N})}{\sum_{n=1}^N\phieps(x^1-X_{t,n,N}^1)}\right|\right]\\\leq&
\frac1\alpha\E\left[\sup_{t\in[0,T]}\left|\frac1N\sum_{n=1}^N
\phieps(x^1-\bar X_{t,n}^1)\partial_1V(\bar X_{t,n})-\phieps*p_t^{\eta,\partial_1V}(x^1)\right|\right]\\
&\quad+\frac1\alpha\E\left[\sup_{t\in[0,T]}\left|\frac1N\sum_{n=1}^N\phieps(x^1-\bar X_{t,n}^1)-\phieps*p_t^{\eta,1}(x^1)\right|\right]+\frac K{\alpha\varepsilon^2N}\E\left[\sup_{t\in[0,T]}\sum_{n=1}^N\left|X_{t,n}-\bar
X_{t,n}\right|\right]\\
\leq&\frac1{\sqrt N}e^{\frac
  K{\alpha\varepsilon^2}}.
\end{align*}
Indeed, $(\phieps(x^1-\bar X_{t,n}^1)\partial_1V(\bar
X_{t,n})-\phieps*p_t^{\eta,\partial_1V}(x^1))_{n\in\{1\hdots N\}}$, as
well as $(\phieps(x^1-\bar
X_{t,n}^1)-\phieps*p_t^{\eta,1}(x^1))_{n\in\{1\hdots N\}}$, are i.i.d. centered random variables whose variance is
bounded by $\frac K{\varepsilon^2}$, uniformly in time.
\end{proof}

\section{Numerical results}\label{numerique}

In this section we give some numerical simulations to illustrate our
previous results. Here, the parameter $\alpha$, which was introduced to
enable theoretical estimations, is taken to be $0$.

Notice that the simulation used here is not exactly the one actually
used in applications. In particular, in those simulations time averages
are used in order to smooth  the problem : first the equation on
$A_t$ given in \eqref{eq:EDS_ABF} can be replaced by $$\partial_tA_t'(z)=\frac1\tau\Large(\E[F(X_t)|\xi(X_t)=z]-A_t\Large)$$ which
makes $A_t$ vary more smoothly. Second, one can use, in addition of the
particle approximation, an ergodic average for the computation of the
conditional expectation in \eqref{eq:EDS}.

In order to accelerate the convergence, one can also use a selection mechanism that gives more weight to
particles located in less explored areas (see \cite{lelievre-rousset-stoltz-07}).

\subsection{Efficiency of the ABF method}

Let us introduce a low dimensional example to illustrate the efficiency of
the ABF method and its particle approximation.

In this first example, we simulated the particle approximation with $1000$
particles, in the potential defined for $(x,y)$ in $[-2,2]\times\R$ by
\begin{equation}\label{eq:defpot1}
V_1(x,y)=5e^{-x^2-y^2}-5e^{-(x-1)^2-y^2}-5e^{-(x+1)^2-y^2}+0.2x^4+0.2y^4,
\end{equation} and extended periodically in the $x$ direction with period $4$.
The level sets of $V_1$ are depicted on Figure \ref{fig:exemple_part}.

On Figure \ref{fig:exemple_part}, we also plotted the position of the
particles after $2000$ iteration of an Euler-Maruyama approximation of
Equation \eqref{eq:EDS_N_eps} with a time step of $0.01$. The value of the
parameters are $\varepsilon=0.01$, $\beta=10$ and $N=1000$. On Figure
\ref{fig:exemple_force}, we plotted the graph of the mean force
(computed by numerical integration, which is still possible
due to the low dimensionnality), and the value of the approximate mean
force computed on a regular grid. The $\LL^1-$distance between the two
functions is $6.93\times10^{-2}$, while the $\LL^1-$norm of the function
$A'$ is 12.9.

Notice that without biasing force, one obtains a very poor sampling,
since the particles do not escape from the well they started in: see Figure
\ref{fig:sans_ABF}, where we plotted $200$ independent simulations of a
Langevin dynamics \eqref{eq:langevin} using $2000$ iterations of an
Euler-Maruyama scheme of time step $0.01$.

\begin{figure}[htbp]
\centerline{\epsfig{file=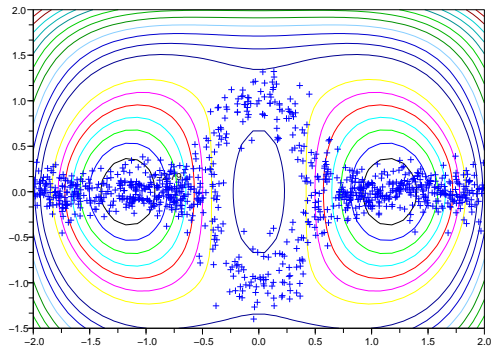,width=80mm,angle=0}}
\caption{Contour plot of the potential $V_1$ with the positions of
  $1000$ particles at time $t=20$.}\label{fig:exemple_part}
\end{figure}

\begin{figure}[htbp]
\centerline{\epsfig{file=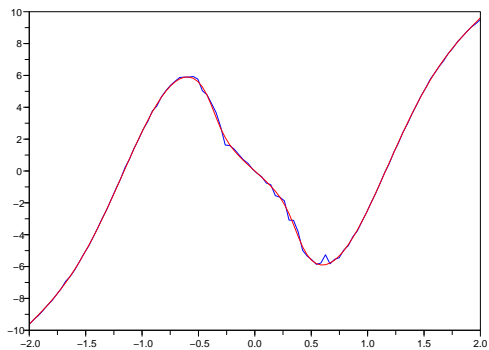,width=80mm,angle=0}}
\caption{Particle approximation of the mean force. The smooth curve is
  the actual value of the mean force, the rough one is the approximation.}\label{fig:exemple_force}
\end{figure}

\begin{figure}[htbp]
\centerline{\epsfig{file=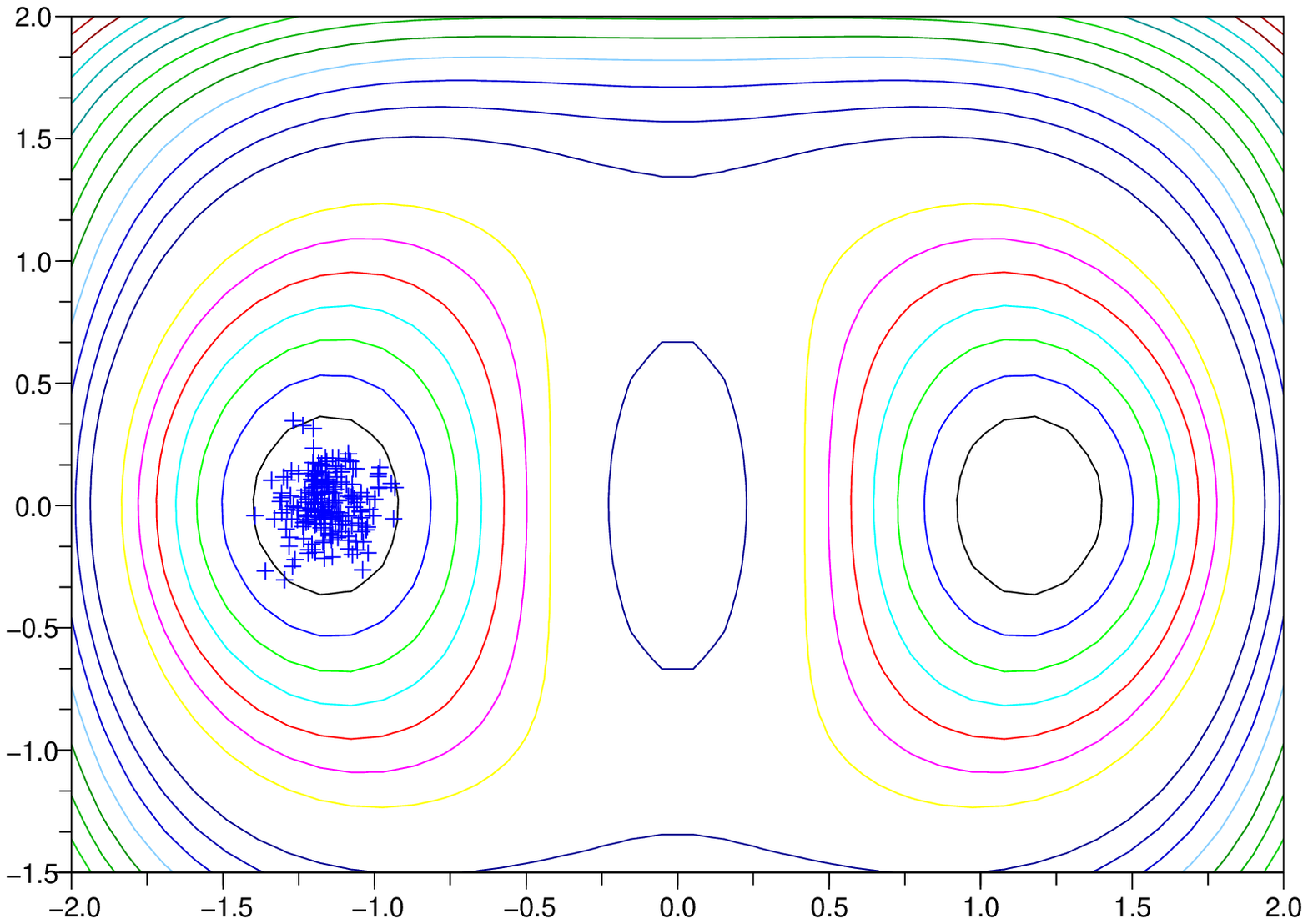,width=80mm,angle=0}}
\caption{$200$ independent realizations of a Langevin dynamics at time $t=20$.}\label{fig:sans_ABF}
\end{figure}

On Figure \ref{fig:decrN} we show the $\LL^1$ distance between the
actual value of the mean force $A'$ and its approximation at time $20$,
obtained for one simulation of the system,
as a function of the number of particles used in the simulation. Using a
least square regression, we find that the slope of the curve is approximatively
$-0.59$, which matches with the theoretical rate of $N^{-1/2}$.

\begin{figure}[htbp]
\centerline{\epsfig{file=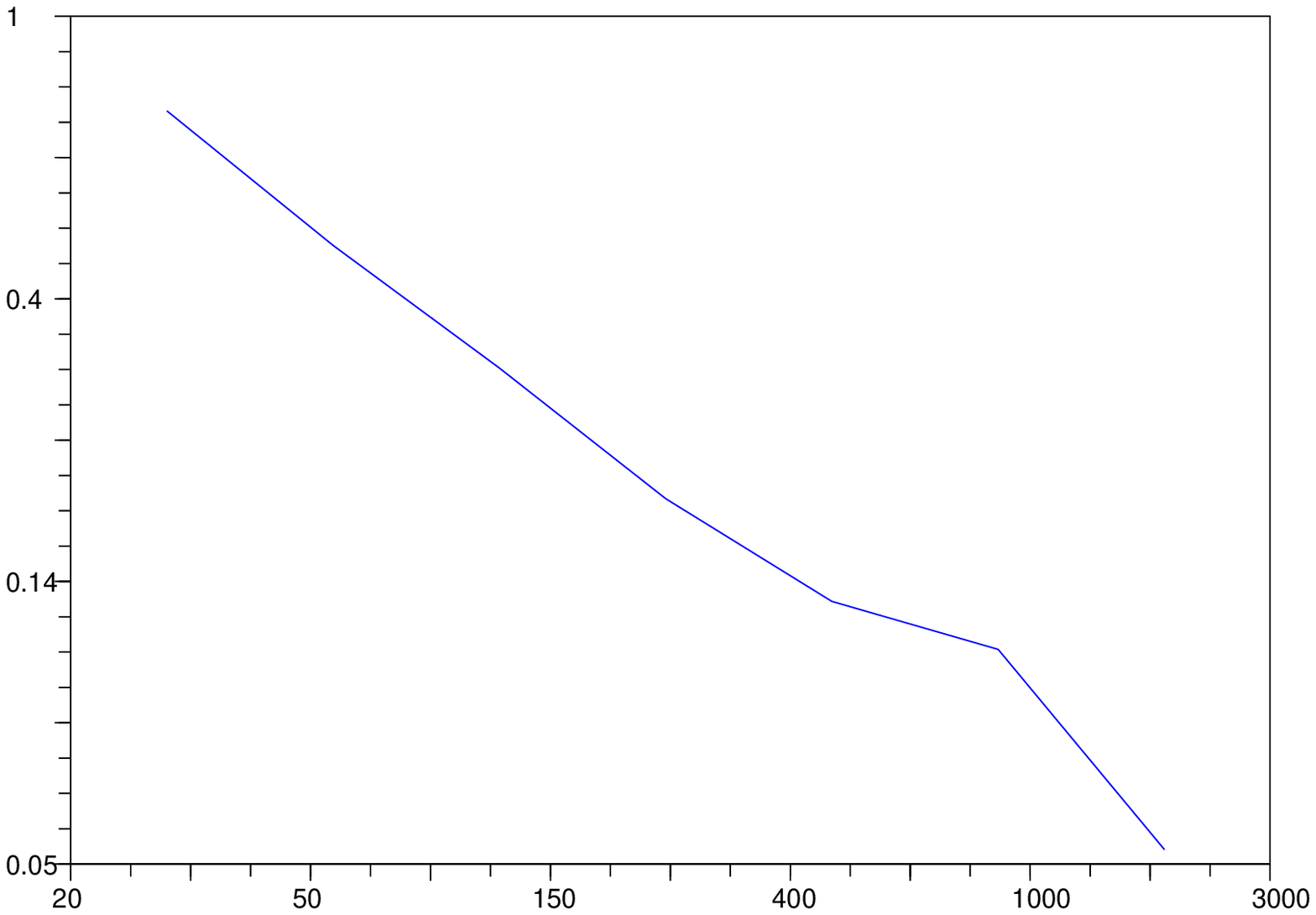,width=80mm,angle=0}}
\caption{Error as a function of $N$ (logarithmic scale).}\label{fig:decrN}
\end{figure}

\subsection{Tuning of the parameters}
In Theorem \ref{th:conv_N}, we showed that the particle approximation
converges as $\varepsilon$ goes to $0$ and $N$ goes to infinity, provided
that $\varepsilon$ does not go to zero too fast compared to $N$. The
practical difficulty that one encounters to apply this result is to choose
a good scaling for $\varepsilon$ in term of $N$.

On Figure \ref{fig:erreps}, we can see the $\LL^1$ error between the
mean force and its aproximation at time $20,$ as a function of the
parameter $\varepsilon,$ using $N=1500$ particles.

\begin{figure}[htbp]
\centerline{\epsfig{file=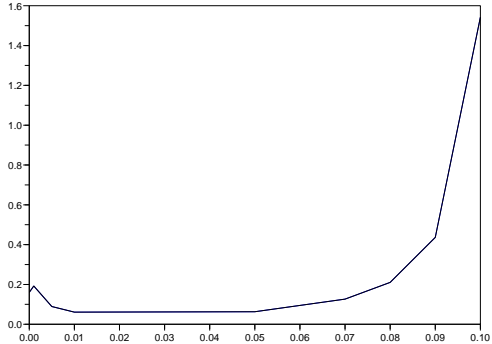,width=80mm,angle=0}}
\caption{Error as a function of $\varepsilon$.}\label{fig:erreps}
\end{figure}

Actually, for a fixed value of $N$, there is only a small range of values
for $\varepsilon$ for which the error is small. 

First, the limit of the
error as $\varepsilon$ goes to $0$ does not even vanish as $N$ tends to
infinity. The reason is that, since the particles interact with each
other in a range of $\varepsilon,$ the number of particles which
interact with a given particle is of order $\varepsilon N$. Hence, when
$\varepsilon$ tends to $0$ while $N$ is fixed, the particles cannot see
each other. Therefore, the natural limit of the particle system in the limit
$\varepsilon\rightarrow0$, $N$ fixed, should be a system of independent
particles following the dynamics $$\dx X_t=\left(-\nabla V(X_t)+e_1\partial_1V(X_t)\right)\dx
t+\sqrt{2\beta^{-1}}\dx W_t.$$ Unfortunately,
in the general case, the drift in the above
dynamics is not obtained as the gradient of a potential, so that no invariant measure for $X_t$ is known. This would
consequently induce a non vanishing bias in the estimation of $A$.

For example, for the potential $V(x,y)=\frac12(y-\sin(2\pi x))^2$, one can prove that
the dynamics obtained by canceling the force on the reaction coordinate
$x$, namely the couple $(\{X_t\},Y_t)$ defined by the dynamics
\begin{align*}\begin{cases}
\dx X_t&=\sqrt2\dx W_t^1,\\
\dx Y_t&=\left(-Y_t+\sin(2\pi X_t)\right)\dx t+\sqrt 2\dx W_t^2
\end{cases}\end{align*}
converges in law to the couple
$\left(\xi,\int_0^\infty e^{-s}\sin(2\pi(\xi+\sqrt2W_s))\dx s+G\right)$,
where $W$ is a standard Brownian motion, $\xi$ is uniformly distributed on $\T$, and $G$ is a standard
normal random variable, independent of $W$. This is not the correct
limit distribution, since the law of $Y$ conditionned to the value of
$\{X\}$ should be Gaussian, which is not the case here.

For a large value of $\varepsilon,$ the behavior of the particle system
can be really different from the expected behavior of the dynamic
\eqref{eq:EDS_ABF}. In the following example, the particles, instead of
freely visiting the $x$ axis, keep stuck in the local minima they
started in. Indeed, the large value of $\varepsilon$ made that the
biasing term is close to the mean of $\partial_1V(X^i)$ on all
particles, whose value is close to $0$. Consequently, the biasing force
is not large enough to prevent the particle from being trapped in the
local minima.

In the following example we considered the potential $V_1$ defined in
\eqref{eq:defpot1}, took $\varepsilon=1$, and simulated $200$
particles during $2000$ iterations of time step $0.01$. The result can
be seen on Figure~\ref{fig:eps_grand}.

\begin{figure}[htbp]
\centerline{\epsfig{file=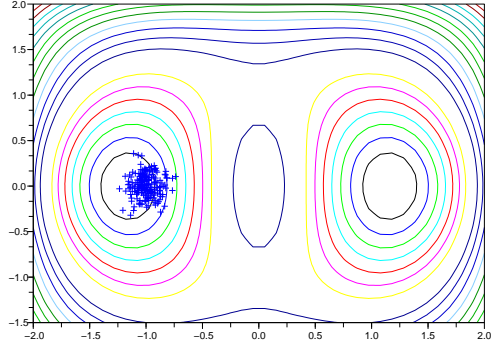,width=80mm,angle=0}}
\caption{Bad sampling due to a too large value of $\varepsilon$.}\label{fig:eps_grand}
\end{figure}

One way to increase the sample size while keeping the number $N$ of
particles fixed is to include time averages for the estimation of the
conditional expectation. This is actually the common practice in the
applied community (see \cite{henin-chipot-04,chipot-pohorille-07}).

\subsection{Discussion on the choice of the reaction coordinate}
We now give another example to illustrate the limitations of the ABF method. 
We consider the 4-periodical potential (in the $x$-direction) defined for $(x,y)$ in $[-2,2]\times\R$ by
\begin{equation}\label{eq:defpot2}
V_2(x,y)=3e^{-x^2-(y-1/3)^2}-3e^{-x^2-(y-5/3)^2}-5e^{-(x-1)^2-y^2}-5e^{-(x+1)^2-y^2}+0.2x^4+0.2(y-1/3)^4,
\end{equation}
whose level sets are depicted on Figure
\ref{fig:contre_exemple_part}. This potential has been introduced in \cite{metzner-schutte-vandeneijnden-06}.

The potential $V_2$ displays two deep minima approximately located at
$(\pm1,0)$. There is a maximum located at $(0,0.5)$, so that there are
two possible paths between the main minima. The first one is a direct
path meeting a saddle point approximately at $(0,-0.3)$. The other path
 goes through two saddle points at $(\pm0.5,1)$ and a small minima at
$(0,1.5)$. Even if the first path is more direct than the second one,
the prefered path in low temperature regimes will be the second one, since its energy barrier is
smaller.

We simulated the particle approximation of the ABF method with $N=1000$ particles, window width
$\varepsilon=0.01$, after $2000$ iterations of an Euler-Maruyama scheme
of time step $0.01$, and plotted the positions of the particles on
Figure \ref{fig:contre_exemple_part}.

\begin{figure}[htbp]
\centerline{\epsfig{file=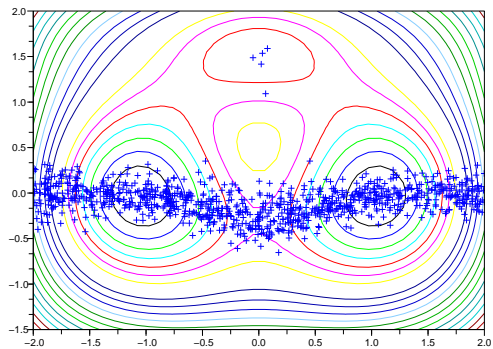,width=80mm,angle=0}}
\caption{Poor sampling due to a bad choice of the
  reaction coordinate.}\label{fig:contre_exemple_part}
\end{figure}

\begin{figure}[htbp]
\centerline{\epsfig{file=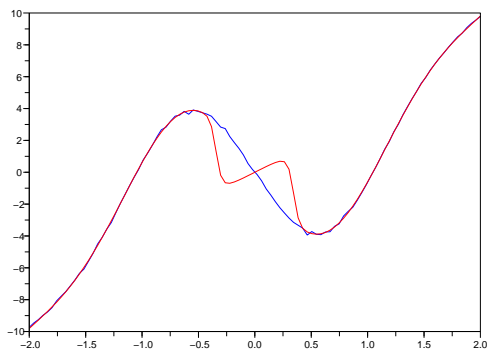,width=80mm,angle=0}}
\caption{Biased evaluation of the biasing force due to a bad choice of
  the reaction coordinate. The smooth curve is the value of the mean
  force. The rough curve is the approximation. Here, the approximation
  does not see the variations of the mean force around $0$.}\label{fig:contre_exemple_force}
\end{figure}

At the low temperature $\beta=10,$ the particles are expected to hop
from one well to the other
mainly through the upper channel, which is not the case here. This is
due to a bad choice of the reaction coordinate. Indeed, the biasing
force only acts in the $x$ direction, so that a particle trapped in the
left side well will naturally escape through a horizontal path, and will
take the lower channel. As a result,
the computation of the force is clearly biased, because of the poor
sampling of the upper channel, see Figure
\ref{fig:contre_exemple_force}, the $\LL^1-$distance between the two
functions is of 0.4.

We still have convergence to the correct mean force, but at a slow rate,
since the reaction coordinate has not been chosen in an optimal
way. Indeed, with the same parameters, but after $2.10^6$ iterations,
the result is much better, see Figures \ref{fig:exemple2_part} and
\ref{fig:exemple2_force}. The $\LL^1-$distance between the mean force
and its approximation is of 0.15, while the function $A'$ has
$\LL^1-$norm 10.9.

\begin{figure}[htbp]
\centerline{\epsfig{file=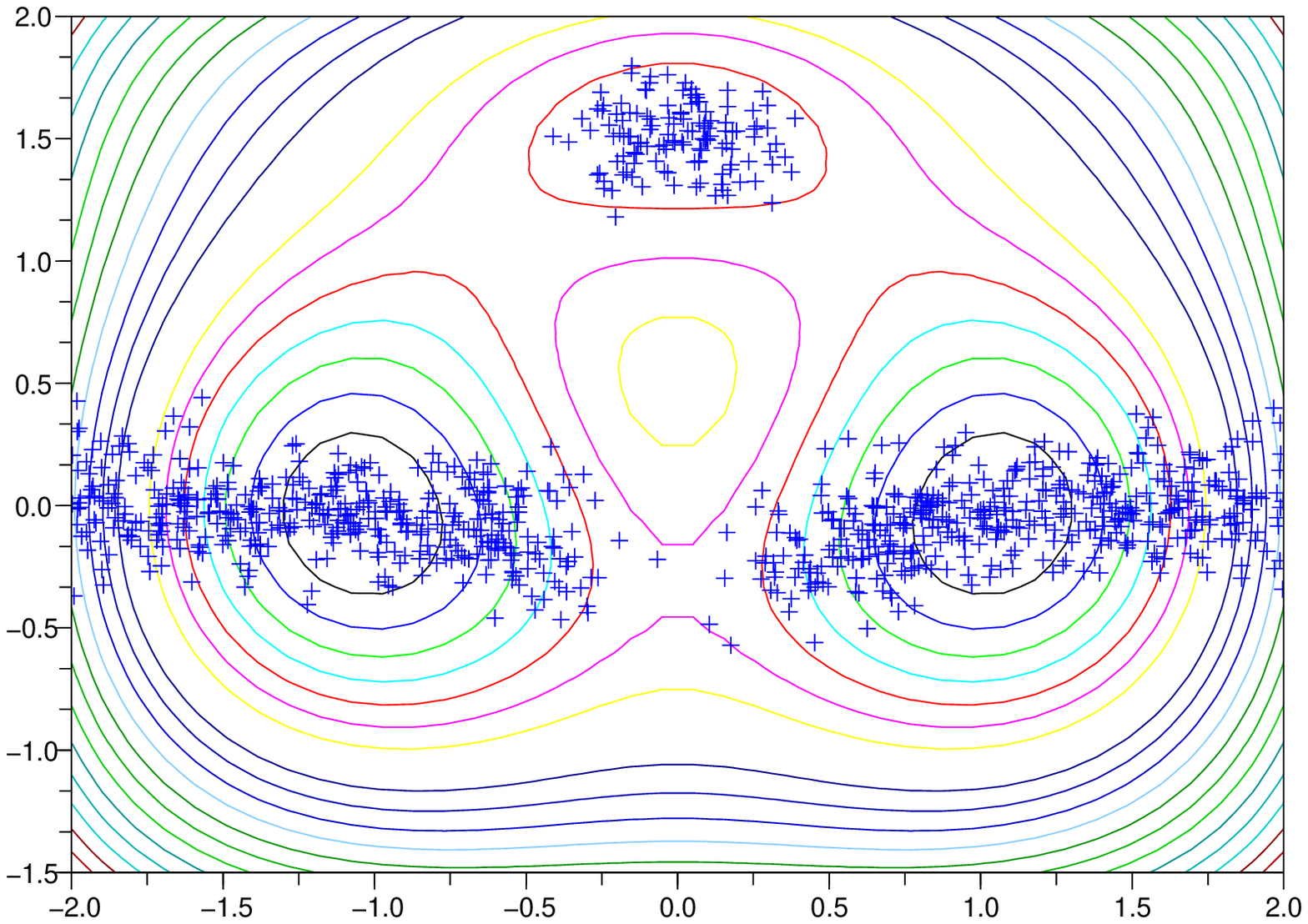,width=80mm,angle=0}}
\caption{Same simulation as on Figure \ref{fig:contre_exemple_part} at time 2000.}\label{fig:exemple2_part}
\end{figure}

\begin{figure}[htbp]
\centerline{\epsfig{file=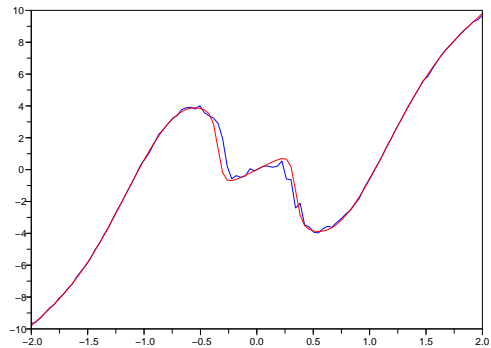,width=80mm,angle=0}}
\caption{Approximation of the free energy corresponding to Figure
  \ref{fig:exemple2_part}. The smooth curve is the free energy, the
  rough one is the approximation.}\label{fig:exemple2_force}
\end{figure}

\bibliography{biblio}

\begin{thebibliography}{10}

\bibitem{adams-78}
R.~Adams.
\newblock {\em Sobolev spaces}.
\newblock Academic Press, 1978.

\bibitem{bossy-jabir-talay-09}
M.~Bossy, J.F. Jabir, and Talay D.
\newblock On conditional {M}c{K}ean {L}agrangian stochastic models.
\newblock Preprint available at {\tt http://hal.inria.fr/inria-00345524}.

\bibitem{brezis-83}
H.~Br\'ezis.
\newblock {\em Analyse fonctionnelle. Th\'eorie et applications}.
\newblock Collection Math\'ematiques appliqu\'ees pour la ma\^itrise. Masson,
  Paris, 1983.
\newblock In French.

\bibitem{chipot-pohorille-07}
C.~Chipot and A.~Pohorille, editors.
\newblock {\em Free Energy Calculations}, volume~86 of {\em Springer Series in
  Chemical Physics}.
\newblock Springer, 2007.

\bibitem{darve-pohorille-01}
E.~Darve and A.~Pohorille.
\newblock Calculating free energy using average forces.
\newblock {\em J. Chem. Phys.}, 115:9169--9183, 2001.

\bibitem{dautray-lions-99}
R.~Dautray and P.L. Lions.
\newblock {\em Mathematical Analysis and Numerical Methods for Science and
  Technology}.
\newblock Springer Verlag, 1999.

\bibitem{dermoune-03}
A.~Dermoune.
\newblock Propagation and conditional propagation of chaos for pressureless gas
  equations.
\newblock {\em Probab. Theor. Relat. Fields}, 126:459--479, 2003.

\bibitem{henin-chipot-04}
J.~H\'enin and C.~Chipot.
\newblock Overcoming free energy barriers using unconstrained molecular
  dynamics simulations.
\newblock {\em J. Chem. Phys.}, 121:2904--2914, 2004.

\bibitem{krylov-rockner-05}
N.V. Krylov and M.~R\"ockner.
\newblock Strong solutions of stochastic equations with singular time dependent
  drift.
\newblock {\em Probab. Theor. Relat. Fields}, 131:154--196, 2005.

\bibitem{lelievre-rousset-stoltz-07}
T.~Leli\`evre, M.~Rousset, and G.~Stoltz.
\newblock Computation of free energy profiles with parallel adaptive dynamics.
\newblock {\em J. Chem. Phys}, 126:134111, 2007.

\bibitem{lelievre-rousset-stoltz-08}
T.~Leli\`evre, M.~Rousset, and G.~Stoltz.
\newblock Long-time convergence of an adaptive biasing force method.
\newblock {\em Nonlinearity}, 21:1155--1181, 2008.

\bibitem{lions-69}
J.L. Lions.
\newblock {\em Quelques m\'ethodes de r\'esolution des probl\`emes aux limites
  non-lin\'eaires}.
\newblock Dunod, 1969.
\newblock In French.

\bibitem{lions-magenes-68}
J.L. Lions and E.~Magenes.
\newblock {\em Probl\`emes aux limites non homog\`enes et applications}.
\newblock Dunod, Paris, 1968-1970.
\newblock In French.

\bibitem{metzner-schutte-vandeneijnden-06}
P.~Metzner, Ch. Sch\"utte, and E.~Vanden-Eijnden.
\newblock Illustration of transition path theory on a collection of simple
  examples.
\newblock {\em J. Chem. Phys}, 125:084110, 2006.

\bibitem{sznitman-89}
A.S. Sznitman.
\newblock Topics in propagation of chaos.
\newblock {\em Lecture notes in mathematics}, 1464, 1989.

\bibitem{talay-vaillant-03}
D.~Talay and O.~Vaillant.
\newblock A stochastic particle method with random weights for the computation
  of statistical solutions of {M}c{K}ean-{V}lasov equations.
\newblock {\em Ann. Appl. Probab.}, 13(1):140--180, 2003.

\bibitem{temam-79}
R.~Temam.
\newblock {\em {N}avier-{S}tokes equations and nonlinear functionnal analysis}.
\newblock North Holland, Amsterdam, 1979.

\bibitem{tran-08}
V.C. Tran.
\newblock A wavelet particle approximation for {M}c{K}ean-{V}lasov and
  2{D}-{N}avier-{S}tokes statistical solutions.
\newblock {\em Stochastic Process. Appl.}, 118:284--318, 2008.

\bibitem{tsybakov-04}
A.B. Tsybakov.
\newblock {\em Introduction \`a l'estimation non-param\'etrique}.
\newblock Springer, 2004.
\newblock In French.

\end{thebibliography}
\bibliographystyle{plain}

\end{document}